%% file: main.tex
\documentclass{article}

\usepackage{times}
\usepackage{graphicx} 
\usepackage{subfigure} 

\usepackage{natbib}

\usepackage{algorithm}
\usepackage{algorithmic}
\usepackage{balance}

\usepackage{hyperref}

\usepackage[accepted]{icml2017}

\usepackage{amsmath,amssymb,amsthm}
\usepackage{psfrag}
\input{MACPdef}
\input{MACPdef-ams}

\DeclareMathOperator*{\argmax}{arg\,max}
\DeclareMathOperator*{\argmin}{arg\,min}

\newcommand{\ie}{i.e.\@}
\newcommand{\eg}{e.g.\@}
\newcommand{\te}{\tilde{\epsilon}}
\newcommand{\tv}{\tilde{\vv}}
\newcommand{\tw}{\tilde{\w}}
\newcommand{\normM}[1]{\norm{#1}_{\M_{\lambda}^{-1}}}

\usepackage{mathtools}
\newcommand{\ceil}[1]{\left\lceil #1 \right\rceil}

\usepackage{eqparbox}
\renewcommand\algorithmiccomment[1]{%
  $\left\{ #1 \right\}$%
}

\icmltitlerunning{Efficient coordinate-wise leading eigenvector computation}

\begin{document} 

\twocolumn[
\icmltitle{Efficient coordinate-wise leading eigenvector computation}



\icmlsetsymbol{equal}{*}

\begin{icmlauthorlist}
\icmlauthor{Jialei Wang}{equal,uchicago}
\icmlauthor{Weiran Wang}{equal,ttic}
\icmlauthor{Dan Garber}{ttic}
\icmlauthor{Nathan Srebro}{ttic}
\end{icmlauthorlist}

\icmlaffiliation{uchicago}{University of Chicago, 5801 S Ellis Ave, Chicago, IL 60637}
\icmlaffiliation{ttic}{Toyota Technological Insititute at Chicago, 6045 S Kenwood Ave, Chicago, IL 60637}

\icmlcorrespondingauthor{Nathan Srebro}{nati@ttic.edu}

\icmlkeywords{eigenvalue problem, power method, shift-and-invert, coordinate descent}

\vskip 0.3in
]



\printAffiliationsAndNotice{\icmlEqualContribution} 

\begin{abstract}
We develop and analyze efficient "coordinate-wise" methods for finding the leading eigenvector, where each step involves only a vector-vector product.  We establish global convergence with overall runtime guarantees that are at least as good as Lanczos's method and dominate it for slowly decaying spectrum. Our methods are based on combining a shift-and-invert approach with coordinate-wise algorithms for linear regression. 
\end{abstract}

\input{intro}
\input{algorithm1}
\input{algorithm2}

\input{algorithm3}

\input{experiment}


\cleardoublepage
\balance
\bibliographystyle{icml2017}
\bibliography{cpm}


\cleardoublepage
\appendix
\input{append}

\end{document}

%% file: MACPdef.tex
\newcommand{\A}{\ensuremath{\mathbf{A}}}

\newcommand{\D}{\ensuremath{\mathbf{D}}}

\newcommand{\I}{\ensuremath{\mathbf{I}}}

\newcommand{\M}{\ensuremath{\mathbf{M}}}

\renewcommand{\SS}{\ensuremath{\mathbf{S}}}

\newcommand{\U}{\ensuremath{\mathbf{U}}}

\newcommand{\W}{\ensuremath{\mathbf{W}}}

\renewcommand{\aa}{\ensuremath{\mathbf{a}}}
\renewcommand{\b}{\ensuremath{\mathbf{b}}}
\renewcommand{\c}{\ensuremath{\mathbf{c}}}

\newcommand{\e}{\ensuremath{\mathbf{e}}}

\newcommand{\g}{\ensuremath{\mathbf{g}}}

\newcommand{\p}{\ensuremath{\mathbf{p}}}

\newcommand{\uu}{\ensuremath{\mathbf{u}}}
\newcommand{\vv}{\ensuremath{\mathbf{v}}}
\newcommand{\w}{\ensuremath{\mathbf{w}}}
\newcommand{\x}{\ensuremath{\mathbf{x}}}
\newcommand{\y}{\ensuremath{\mathbf{y}}}
\newcommand{\z}{\ensuremath{\mathbf{z}}}
\newcommand{\0}{\ensuremath{\mathbf{0}}}
\newcommand{\1}{\ensuremath{\mathbf{1}}}

\newcommand{\bbR}{\ensuremath{\mathbb{R}}}

\newcommand{\calO}{\ensuremath{\mathcal{O}}}
\newcommand{\calP}{\ensuremath{\mathcal{P}}}

\newcommand{\abs}[1]{\left\lvert#1\right\rvert}
\newcommand{\norm}[1]{\left\lVert#1\right\rVert}

{%
\begin{list}{#1}{
\vspace{-\topsep}
\vspace{-\partopsep}
\setlength{\itemindent}{0cm}
\setlength{\rightmargin}{0cm}
\setlength{\listparindent}{0cm}
\settowidth{\labelwidth}{#1}
\setlength{\leftmargin}{\labelwidth}
\addtolength{\leftmargin}{\labelsep}
\setlength{\itemsep}{0cm}
}%
}%
{%
\end{list}
\vspace{-\topsep}
\vspace{-\partopsep}
}

{\begin{enumerate}%
}%
{\end{enumerate}}



%% file: MACPdef-ams.tex
\newtheorem{thm}{Theorem}
\newtheorem{lem}[thm]{Lemma}

\theoremstyle{remark}

\newtheorem*{rmk*}{Remark}


%% file: intro.tex
\section{Introduction}
\label{sec:intro}

Extracting the top eigenvalues/eigenvectors of a large symmetric matrix is a fundamental step in various machine learning algorithms. One prominent example of this problem is principal component analysis (PCA), in which we extract the top eigenvectors of the data covariance matrix, and there has been continuous effort in developing efficient stochastic/randomized algorithms for large-scale PCA (e.g., \citealp{Garber_16a, Shamir15, Allen-ZhuLi16a}). 
The more general eigenvalue problems for large matrices without the covariance structure is relatively less studied. The method of choice for this problem has been the power method, or the faster but often less known Lanczos algorithm~\citep{GolubLoan96a}, which are based on iteratively computing matrix-vector multiplications with the input matrix until the component of the vector that lies in the tailing eigenspace vanishes. However, for very large-scale and dense matrices, even computing a single matrix-vector product is expensive. 

An alternative is to consider much cheaper vector-vector products, \ie, instead of updating all the entries in the vector on each iteration by a full matrix-vector product, we consider the possibility of only updating one coordinate, by only computing the inner product of single row of the matrix with the vector. Such operations do not even require to store the entire matrix in memory. Intuitively, this may result in an overall significant speedup, in certain likely scenarios, since certain coordinates in the matrix-vector product are more valuable than others for making local progress towards converging to the leading eigenvector.
Indeed, this is the precise rational behind coordinate-descent methods that were extensively studied and are widely applied to convex optimization problems, see~\citet{Wright15a} for a comprehensive survey. Thus, given the structure of the eigenvalue problem which is extremely suitable for coordinate-wise updates, and the celebrated success of coordinate-descent methods for convex optimization, a natural question is whether such updates can be applied for eigenvector computation with provable global convergence guarantees, despite the inherent non-convexity of the problem.

Recently,~\citet{Lei_16a} have proposed two such methods, Coordinate-wise Power Method (CPM) and Symmetric Greedy Coordinate Descent (SGCD). Both methods update on each iteration only $k$ entries in the vector, for some fixed $k$. CPM updates on each iteration the $k$ coordinates 
that would change the most under one step of power method, 
 while SGCD applies a greedy heuristic for choosing the coordinates to be updated. The authors show that CPM enjoys a global convergence rate, with rate similar to the classical power iterations algorithm provided that $k$, the number of coordinates to be updated on each iteration, is sufficiently large (or equivalently, the "noise" outside the $k$ selected coordinate is sufficiently small). In principle, this might force $k$ to be as large as the dimension $d$, and indeed in their experiments they set $k$ to grow linearly with $d$, which is overall not significantly faster than standard power iterations, and does not truly capture the concept of coordinate updates proved useful to convex problems. The second algorithm proposed in~\citet{Lei_16a}, SGCD, is shown to converge \textit{locally} with a linear rate already for $k=1$ (i.e., only a single coordinate is updated on each iteration), however this results assume that the method is initialized with a vector that is already sufficiently close (in a non-trivial way) to the leading eigenvector. The dependence of CPM and SGCD on the 
inverse relative eigengap\footnote{Defined as $\frac{\sigma_1 (\A)}{\sigma_1 (\A) - \sigma_2 (\A)}$ for a positive semidefinite matrix $\A$, where $\sigma_i (\A)$ is the $i$-th largest eigenvalue of $\A$.} 
of the input matrix is similar to that of the power iterations algorithm, i.e. linear dependence, which in principal is suboptimal, since square-root dependence can be obtained by methods such as the Lanczos algorithm.

We present \textit{globally-convergent} coordinate-wise algorithms for the leading eigenvector problem which resolves the abovementioned concerns and significantly improve over previous algorithms. 
Our algorithms update only a single entry at each step and enjoy linear convergence. Furthermore, for a particular variant, the convergence rate depends only on the square-root of the inverse relative eigengap, 
yielding a total runtime that dominates that of the standard power method and competes with that of the Lanczos's method. 
In Section~\ref{sec:shift-and-invert}, we discuss the basis of our algorithm, the \textit{shift-and-invert power method}~\citep{GarberHazan15c, Garber_16a}, which transforms the eigenvalue problem into a series of convex least squares problems. 
In Section~\ref{sec:coord-descent}, we show the least squares problems can be solved using efficient coordinate descent methods that are well-studied for convex problems. This allow us to make use of principled coordinate selection rules for each update, all of which have established convergence guarantees.

We provide a summary of the time complexities of different globally-convergent methods in Table~\ref{t:complexity}. In particular, it is observable that in cases where either the spectrum of the input matrix or the magnitude of its diagonal entries is slowly decaying, 
our methods can yield provable and significant improvements over previous methods. 
For example, for a spiked covariance model whose eigenvalues are 
$\rho_1 > \rho_1 - \Delta = \rho_2=\rho_3=\dots$, our algorithm (SI-GSL) can have a runtime independent of the eigengap $\Delta$, while the runtime of Lanczos's method depends on $\frac{1}{\sqrt{\Delta}}$. 
We also verify this intuition empirically via numerical experiments.


\begin{table}
\centering
\caption{Time complexity (total number of coordinate updates) for finding an estimate $\w$ satisfying $(\w^{\top}\p_1)^2 \geq 1-\epsilon$ with at least constant probability, where $\p_1$ is the leading eigenvector of a positive semidefinite matrix $\A \in \bbR^{d\times d}$, with eigenvalues $\rho_1,\rho_2,\ldots$ in descending order. 
Since all methods are randomized, we assume all are initialized with a random unit vector.}
\label{t:complexity}
\begin{tabular}{|c|c|} \hline
Method & Time complexity \\ \hline 
Power method & $\calO \left( \frac{\rho_1}{\rho_1 - \rho_2} \cdot d^2 \cdot \log \frac{d}{\epsilon} \right)$ \\
Lanczos & $\calO \left( \sqrt{ \frac{\rho_1}{\rho_1 - \rho_2} }  \cdot d^2 \cdot \log \frac{d}{\epsilon} \right)$  \\
Ours (SI-GSL) & $\calO \left(  \frac{\rho_1 - \frac{1}{d} \sum_{i=1}^d \rho_i}{ \rho_1 - \rho_2 } \cdot d^2 \cdot \log \frac{d}{\epsilon} \right)$  \\ 
Ours (SI-ACDM) & $\calO \left(   \frac{ \frac{1}{d} \sum_{i=1}^d \sqrt{ \rho_1 -  \A_{ii} }}{ \sqrt{ \rho_1 - \rho_2 } } \cdot d^2 \cdot \log \frac{d}{\epsilon} \right)$  \\
\hline
\end{tabular}
\end{table}

\paragraph{Notations} We use boldface uppercase letters (\eg, $\A$) to denote matrices, and boldface lowercase letters (\eg, $\x$) to denote vectors. For a positive definite matrix $\M$, the vector norm $\norm{\cdot}_{\M}$ is defined as $\norm{\w}_{\M} = \sqrt{ \w^\top \M \w } = \norm{\M^{\frac{1}{2}} \w }$ for any $\w$. We use $\A[ij]$ to denote the element in row $i$ and column $j$ of the matrix $\A$, and use $\x[i]$ to denote the $i$-th element of the vector $\x$ unless stated otherwise. Additionally, $\A[i:]$ and $\A[:j]$ denote the $i$-th row and $j$-th column of the matrix $\A$ respectively.

\paragraph{Problem formulation} 
We consider the task of extracting the top eigenvector of a symmetric positive definite matrix $\A \in \bbR^{d}$ (extensions to other setting are discussed in Section~\ref{sec:algorithm-extension}). Let the complete set of eigenvalues of $\A$ be $\rho_1 \ge \rho_2 \ge \dots \ge \rho_d \ge 0$, with corresponding eigenvectors $\p_1, \dots, \p_d$ which form an orthonormal basis of $\bbR^d$. Without loss of generality, we assume $\rho_1 \le 1$ (which can always be obtained by rescaling the matrix). Furthermore, we assume the existence of a positive eigenvalue gap $\Delta:= \rho_1 - \rho_2 > 0$ so that the top eigenvector is unique.

%% file: algorithm1.tex
\section{Shift-and-invert power method}
\label{sec:shift-and-invert}

In this section, we introduce the shift-and-invert approach to the eigenvalue problem and review its analysis, which will be the basis for our algorithms.

The most popular iterative algorithm for the leading eigenvalue problem is the power method, which iteratively performs the following matrix-vector multiplications and normalization steps
\begin{align*}
  \tw_{t} \leftarrow \A \w_{t-1}, \qquad \w_{t} \leftarrow  \frac{\tw_{t}}{\norm{\tw_{t}}}, \qquad \text{for} \quad t=1,\dots.
\end{align*}
It can be shown that the iterates become increasingly aligned with the eigenvector corresponding to the largest eigenvalue in magnitude, and the number of iterations needed to achieve 
$\epsilon$-suboptimality in alignment is $\calO\left( \frac{\rho_1}{\Delta} \log \frac{\vert{\w_0^{\top}\p_1}\vert}{\epsilon} \right)$~\citep{GolubLoan96a}\footnote{We can always guarantee that $\vert{\w_0^{\top}\p_1}\vert = \Omega(1/\sqrt{d})$ by taking $\w_0$ to be a random unit vector.}. We see that the computational complexity depends linearly on $\frac{1}{\Delta}$, and thus power method converges slowly if the gap is small.

Shift-and-invert~\citep{GarberHazan15c, Garber_16a} can be viewed as a pre-conditioning approach which improves the dependence of the time complexity on the eigenvalue gap. 
The main idea behind this approach is that, instead of running power method on $\A$ directly, we can equivalently run power method on the matrix $(\lambda \I - \A)^{-1}$ where $\lambda > \rho_1$ is a shifting parameter. Observe that $(\lambda \I - \A)^{-1}$ has exactly the same set of eigenvectors as $\A$, and its eigenvalues are
\begin{align*}
  \beta_1 \ge \beta_2 \ge \dots \ge \beta_d > 0, \qquad \text{where}\quad \beta_i = \frac{1}{\lambda - \rho_i}.
\end{align*}
If we have access to a $\lambda$ that is slightly larger than $\rho_1$, and in particular if $\lambda-\rho_1 = \calO (1) \cdot \Delta$, then the inverse relative eigengap of $(\lambda \I - \A)^{-1}$ is $\frac{\beta_1}{\beta_1-\beta_2} = O(1)$, which means that the power method, applied to this shift and inverted matrix, will converge to the top eigenvector in only a poly-logarithmic number of iterations, in particular, without linear dependence on $1/\Delta$.

In shift-and-invert power method, the matrix-vector multiplications have the form 
$\tw_{t} \leftarrow (\lambda \I - \A)^{-1} \w_{t-1}$, which is equivalent to solving the convex least squares problem 
\begin{align} \label{e:lsq}
  \tw_{t} \leftarrow \argmin_{\w}\; \frac{1}{2} \w^\top (\lambda \I - \A) \w - \w_{t-1}^\top \w.
\end{align}
Solving such least squares problems exactly could be costly itself if $d$ is large. 
Fortunately, power method with approximate matrix-vector multiplications still converges, provided that the errors in each step is controlled; analysis of inexact power method together with applications to PCA and CCA can be found in~\citet{HardtPrice14a,Garber_16a,Ge_16a,Wang_16b}.

\begin{algorithm}[ht!]
  \caption{The shift-and-invert preconditioning meta-algorithm for extracting the top eigenvector of $\A$.}
  \label{alg:meta-shift-and-invert}
  \renewcommand{\algorithmicrequire}{\textbf{Input:}}
  \renewcommand{\algorithmicensure}{\textbf{Output:}}
  \begin{algorithmic}
    \REQUIRE Data matrix $\A$, an iterative least squares optimizer (LSO), a lower estimate $\tilde{\Delta}$ for $\Delta := \rho_1 - \rho_2$ such that $\tilde{\Delta} \in [c_1 \Delta, \, c_2 \Delta]$, and $\underline{\sigma} =  1 + \frac{1 - c_2}{c_2} \tilde{\Delta}$.
    \STATE Initialize $\tw_0 \in \bbR^{d}, \qquad \w_0 \leftarrow \frac{\tw_0}{\norm{\tw_0}}$ 
    \STATE \textbf{// Phase I: locate a $\lambda = \rho_1 + \calO(1) \cdot \Delta$}
    \STATE $s \leftarrow 0, \qquad \lambda_{(0)} \leftarrow 1 + \tilde{\Delta}$
    \REPEAT 
    \STATE $s \leftarrow s+1$ 
    \STATE \textbf{// Power method on $( \lambda_{(s)} \I - \A )^{-1}$ in crude regime}
    \FOR{$t=(s-1) m_1+1, \dots, s m_1$}
    \STATE Optimize the least squares problem with LSO 
    \begin{align*}
      \min_{\w}\ f_t(\w) := \frac{1}{2} \w^\top (\lambda_{(s-1)} \I - \A) \w - \w_{t-1}^\top \w,
    \end{align*}
    with initialization $\frac{ \w_{t-1} }{\w_{t-1}^\top (\lambda \I - \A) \w_{t-1}}$, and output an approximate solution $\tw_t$ satisfying \\
    $\qquad\qquad f_t (\tw_t) \le \min_{\w} f_t (\w) + \epsilon_t$.
    \vspace*{1ex}
    \STATE Normalization: $\w_{t} \leftarrow \frac{\tw_t}{\norm{\tw_t}}$
    \ENDFOR
    \STATE \textbf{// Estimate the top eigenvalue of $(\lambda_{(s)} \I - \A)^{-1} $} 
    \STATE Optimize the least squares problem with LSO 
    \begin{align*}
      \min_{\uu}\ l_s (\uu) := \frac{1}{2}
      \uu^\top
      (\lambda_{(s)} \I - \A) \uu - 
      \w_{s m_1}^\top \uu
    \end{align*}
    with initialization $\frac{ \w_{s m_1} }{\w_{s m_1}^\top (\lambda \I - \A) \w_{s m_1}}$, and output an approximate solution $\uu_s$ satisfying \\
    $\qquad\qquad l_s (\uu_s) \le \min_{\uu} l_s (\uu) + \te_s$.
    \vspace*{1ex}
    \STATE Update: $\Delta_s \leftarrow \frac{1}{2} \cdot \frac{1}{\w_{s m_1}^\top \uu_s - \underline{\sigma}/8}, \;\;
    \lambda_{(s)} \leftarrow \lambda_{(s-1)} - \frac{\Delta_s}{2}$
    \UNTIL {$\Delta_{(s)} \le \tilde{\Delta}}$
    \STATE $\lambda_{(f)} \leftarrow \lambda_{(s)}$
    \STATE \textbf{// Power method on $( \lambda_{(f)} \I - \A )^{-1}$ in accurate regime}
    \FOR{$t = s m_1 + 1, \dots, s m_1 + m_2$}
    \STATE Optimize the least squares problem with LSO 
    \begin{align*}
      \min_{\w}\ f_t(\w) := \frac{1}{2} \w^\top (\lambda_{(f)} \I - \A) \w - \w_{t-1}^\top \w,
    \end{align*}
    with initialization $\frac{ \w_{t-1} }{\w_{t-1}^\top (\lambda \I - \A) \w_{t-1}}$, and output an approximate solution $\tw_t$ satisfying \\
    $\qquad\qquad  f_t (\tw_t) \le \min_{\w} f_t (\w) + \epsilon_t$.
    \vspace*{1ex}
    \STATE Normalization: $\w_{t} \leftarrow \frac{\tw_t}{\norm{\tw_t}}$
    \ENDFOR
    \ENSURE $\w_{s m_1 + m_2}$ is the approximate eigenvector.
  \end{algorithmic}
\end{algorithm}

We give the inexact shift-and-invert preconditioning power method in Algorithm~\ref{alg:meta-shift-and-invert}, which consists of two phases. In Phase I, starting from an lower estimate of $\tilde{\Delta}$ and an upper bound of $\rho_1$, namely $\lambda_{(0)}=1+\tilde{\Delta}$ (recall that by assumption, $\rho_1 \leq 1$), the \textbf{repeat-until} loop locates an estimate of the eigenvalue $\lambda_{(s)}$ such that $ 0 < \lambda_{(s)} - \rho_1 = \Theta(1) \cdot \Delta$, and it does so by estimating the top eigenvalue $\beta_1$ of $(\lambda \I - \A)^{-1}$ (through a small number of power iterations in the \textbf{for} loop), and shrinking the gap between $\lambda_{(s)}$ and $\rho_1$. In Phase II, the algorithm fixes the shift-and-invert parameter $\lambda_{(s)}$, and runs many power iterations in the last \textbf{for} loop to achieve an accurate estimate of the top eigenvector $\p_1$. 

We now provide convergence analysis of Algorithm~\ref{alg:meta-shift-and-invert} following that of~\citet{GarberHazan15c} closely, but improving their rate (removing an extra $\log \frac{1}{\epsilon}$ factor) with the warm-start strategy for least squares problems by~\citet{Garber_16a}, and making the initial versus final error ratio explicit in the exposition. 
We discuss the least-squares solver in Algorithm~\ref{alg:meta-shift-and-invert} in the next section. 

\textbf{Measure of progress} Since $\{\p_1,\dots,\p_d\}$ form an orthonormal basis of $\bbR^d$, we can write each normalized iterate as a linear combination of the eigenvectors as
\begin{align*}
  \w_t = \sum_{i=1}^d \xi_{ti} \p_i, \quad \text{where}\;  \xi_{ti}= \w_t^\top \p_i , \quad \text{for}\; i=1,\dots,d,
\end{align*}
and $\sum_{i=1}^d \xi_{ti}^2 = 1$. Our goal is to have high alignment between the estimate $\w_t$ and $\p_1$, \ie, $\xi_{t1} = \w_t^\top \p_1 \ge 1 - \epsilon$ for $\epsilon\in (0, 1)$. Equivalently, we would like to have $\rho_1 - \w_t^\top \A \w_t \le  \rho_1 \epsilon$ (see Lemma~\ref{lem:alignment-to-objective} in Appendix~\ref{sec:auxiliary}).

\subsection{Iteration complexity of inexact power method} 
\label{sec:shift-and-invert-inexact-power-method}

Consider the inexact shift-and-invert power iterations: 
$\tw_{t} \approx  \argmin_{\w}\; f_t (\w) = \frac{1}{2} \w^\top (\lambda \I - \A) \w - \w_{t-1}^\top \w$, $\w_{t} \leftarrow \frac{\tw_{t}}{\norm{\tw_{t}}}$, where $f_t (\tw_{t}) \le \min_{\w} f_t (\w) + \epsilon_t$, for $t=1,\dots$. 
We can divide the convergence behavior of this method into two regimes: in the crude regime, our goal is to estimate the top eigenvalue (as in Phase I of Algorithm~\ref{alg:meta-shift-and-invert}) 
regardless of the existence of an eigengap;
whereas in the accurate regime, our goal is to estimate the top eigenvector (as in Phase II of Algorithm~\ref{alg:meta-shift-and-invert}).

Following~\citet{GarberHazan15c,Garber_16a}, we provide convergence guarantees for inexact shift-and-invert power iterations in Lemma~~\ref{lem:iteration-complexity-inexact-power-method} (Appendix~\ref{sec:iteration-complexity-inexact-power-method-appendix}), quantifying the sufficient accuracy $\epsilon_t$ in solving each least squares problem for the overall method to converge. 
Note that the iteration complexity for the crude regime is independent of eigengap, 
while in the accurate regime the rate 
in which the error decreases does depend on the eigengap between 
the first two eigenvalues of $(\lambda \I - \A)^{-1}$.

\subsection{Bounding initial error for each least squares}
\label{sec:shift-and-invert-warm-start}

For each least squares problem $f_t (\w)$ in inexact shift-and-invert power method, one can show that the optimal initialization based on the previous iterate 
is~\citep{Garber_16a}
\begin{align*}
  \w_t^{init} = \frac{ \w_{t-1} }{\w_{t-1}^\top (\lambda \I - \A) \w_{t-1}}. 
\end{align*}
We can bound the suboptimality of this initialization, defined as $\epsilon_t^{init} := f_t (\w_t^{init}) - \min_{\w} f_t (\w)$. See full analysis in Appendix~\ref{sec:shift-and-invert-warm-start-append}. 

\begin{lem}  
  \label{lem:ft_ratio}
  Initialize each least squares problem $\min_{\w}\, f_t (\w)$ in inexact shift-and-invert power method from $\w_t^{init} = \frac{ \w_{t-1} }{\w_{t-1}^\top (\lambda \I - \A) \w_{t-1}}$. Then the initial suboptimality $\epsilon_t^{init}$ in $f_t (\w)$ can be bounded by the necessary final suboptimality $\epsilon_t$ as follows. 
  \begin{itemize}
  \item In the crude regime, 
    \begin{align*} 
      \epsilon_t^{init} \le \frac{64 \beta_1^2}{\epsilon^2 \beta_d^2} \left( \frac{(2 \beta_1 / \beta_d)^{T_1} -1}{(2 \beta_1 / \beta_d)-1} \right)^2 \cdot \epsilon_t
    \end{align*}
where $T_1$ is the total number of iterations used (see precise definition in Lemma~\ref{lem:iteration-complexity-inexact-power-method}). 
  \item In the accurate regime, 
    \begin{align*} 
      \epsilon_t^{init} \le \max (G_0,\, 1) \cdot \frac{16 \beta_1^2}{\left( \beta_1 - \beta_2 \right)^2} \cdot \epsilon_t
    \end{align*}
where $G_0$ is the initial condition for shift-and-invert power iterations (see precise definition in Lemma~\ref{lem:iteration-complexity-inexact-power-method}). 
  \end{itemize}
\end{lem}

\subsection{Iteration complexity of Algorithm~\ref{alg:meta-shift-and-invert}}
\label{sec:shift-and-invert-total-iteration-complexity}

Based on the iteration complexity of inexact power method and the warm-start strategy, we derive the total number of iterations (least squares problems) needed by Algorithm~\ref{alg:meta-shift-and-invert}.

\begin{lem}[Iteration complexity of the \textbf{repeat-until} loop in Algorithm~\ref{alg:meta-shift-and-invert}] \label{lem:shift-and-invert-repeat-until}
  Suppose that $\tilde{\Delta} \in [c_1 \Delta,\ c_2 \Delta]$ where $c_2 \le 1$. Set $m_1 = \ceil{ 8 \log \left( \frac{16}{\xi_{01}^2} \right) }$ where $\xi_{01} = \w_0^\top \p_1$, initialize each least squares problem as in Lemma~\ref{lem:ft_ratio}, and maintain the ratio between initial and final error to be $\frac{\epsilon_t^{init}}{\epsilon_t} = \frac{32 \cdot 10^{2 m_1+1}}{\tilde{\Delta}^{2 m_1}}$ and $\frac{\te_s^{init}}{\te_s} = \frac{1024}{\tilde{\Delta}^2}$ throughout the \textbf{repeat-until} loop. 
  Then for all $s \ge 1$ it holds that 
  \begin{align*}
    \frac{1}{2} (\lambda_{(s-1)} - \rho_1) \le \Delta_s \le \lambda_{(s-1)} - \rho_1. 
  \end{align*}
  Upon exiting this loop, the $\lambda_{(f)}$ satisfies
  \begin{align} \label{e:delta-s-interval}
    \rho_1 + \frac{\tilde{\Delta}}{4} \le \lambda_{(f)} \le \rho_1 + \frac{3 \tilde{\Delta}}{2},
  \end{align}
  and the number of iterations by \textbf{repeat-until} is $\calO \left( \log \frac{1}{\tilde{\Delta}} \right)$.
\end{lem}

\begin{lem}[Iteration complexity of the final \textbf{for} loop in Algorithm~\ref{alg:meta-shift-and-invert}]  \label{lem:shift-and-invert-for-loop}
  Suppose that $\tilde{\Delta} \in [c_1 \Delta, c_2 \Delta]$ where $0< c_1 < c_2 \le 1$. Set $m_2 = \ceil{ \frac{1}{2}  \log_{\frac{9}{7}} \left( \frac{G_0^2}{\epsilon} \right) }$ where $G_0 := \frac{ \sqrt{ \sum_{i=2}^d  (\lambda_{(f)} - \rho_i)\cdot (\w_0^\top \p_i)^2 }}{ \sqrt{ (\lambda_{(f)} - \rho_1) \cdot (\w_0^\top \p_1)^2 }}$, initialize each least squares problem as in Lemma~\ref{lem:ft_ratio}, and maintain the ratio between initial and final error to be $\frac{\epsilon_t^{init}}{\epsilon_t} = 100 \max (G_0,\, 1)$ throughout the final \textbf{for} loop. 
  Then the output of Algorithm~\ref{alg:meta-shift-and-invert} satisfies
  \begin{align*}
    \w_{s m_1 + m_2}^\top \p_1 \ge 1 - \epsilon.
  \end{align*}
\end{lem}

As shown in the next section, we will be using linearly convergent solvers for the least squares problems, whose runtime depends on $\log \frac{\epsilon_t^{init}}{\epsilon_t}$. 
Lemma~\ref{lem:shift-and-invert-repeat-until} implies that we need to solve $s m_1 = \calO \left( \log \frac{1}{\Delta} \right)$ least squares problems in Phase I, each with $\log \frac{\epsilon_t^{init}}{\epsilon_t}  = \calO \left( \log \frac{1}{\Delta} \right)$. And Lemma~\ref{lem:shift-and-invert-for-loop} implies that we need to solve $m_2 = \calO \left(  \log \frac{1}{\epsilon} \right)$ least squares problems in Phase II, each with $\log \frac{\epsilon_t^{init}}{\epsilon_t} = \calO \left( 1 \right)$. 

%% file: algorithm2.tex
\section{Coordinate descent for least squares}
\label{sec:coord-descent}

Different from PCA, for general eigenvalue problems, the matrix $\A$ may not have the structure of data covariance. As a result, fast stochastic gradient methods for finite-sums (such as SVRG~\citealp{JohnsonZhang13a} used in~\citealt{Garber_16a}) does not apply. However, we can instead apply efficient coordinate descent (CD) methods for convex problems in our setting. In this section, we review the CD methods and study their complexity for solving the least squares problems created by Algorithm~\ref{alg:meta-shift-and-invert}.

There is a rich literature on CD methods and it has attracted resurgent interests recently due to its simplicity and efficiency for big data problems; we refer the readers to~\citet{Wright15a} for a comprehensive survey. 
In each update of CD, we pick one coordinate and take a step along the direction of negative coordinate-wise gradient. 
To state the convergence properties of CD methods, we need the definitions of a few key quantities. Recall that we would like to solve the least squares problems of the form
\begin{align*}
  \min_{\x}\; f (\x) = \frac{1}{2} \x^\top (\lambda \I - \A) \x - \y^\top \x.
\end{align*}
with the optimal solution $\x^* = (\lambda \I - \A)^{-1} \y$. 

\paragraph{Smoothness}  The gradient $\nabla f (\x) = (\lambda \I - \A) \x - \y$ is coordinate-wise Lipschitz continuous: for all $i=1,\dots,d$, $\x \in \bbR^d$, and $\alpha \in \bbR$, we have 
\begin{align*}
  \abs{ \nabla_i f (\x + \alpha \e_i) - \nabla_i f (\x) } \le L_i \abs{\alpha} \quad \text{for} \; L_i :=  \lambda - \A[ii]
\end{align*}
where $\e_i$ is the $i$-th standard basis. Note that for least squares problems, the coordinate-wise gradient Lipschitz constanst are the diagonal entries of its Hessian. Denote by $L_{\max}:=\max_{1\le i \le d}\; L_i$ and $\bar{L} = \frac{1}{d} \sum_{i=1}^d L_i$ the largest and average Lipschitz constant over all coordinates respectively. These two Lipschitz constants are to be distinguished from the ``global'' smoothness constant $\tilde{L}$, which satisfies
\begin{align*}
\norm{ \nabla f(\x) -  \nabla f(\y) } \le \tilde{L} \cdot \norm{\x - \y},\quad\forall \x, \y. 
\end{align*}
Since $f(\x)$ is quadratic,  $\tilde{L}=\sigma_{\max} \left( \lambda \I - \A \right)$. Observe that~\citep{Wright15a}
\begin{align} \label{e:relation-between-smoothness}
1 \le \tilde{L} / L_{\max} \le d,
\end{align}
with upper bound achieved by a Hessian matrix $\1 \1^\top$.

\paragraph{Strong-convexity} As mentioned before, the matrices $(\lambda \I - \A )$ in our least squares problems are always positive definite, with smallest eigenvalue $\mu := \lambda - \rho_1 = \calO \left( \Delta \right)$. As a result, $f(\x)$ is $\mu$-strongly convex with respect to the $\ell_2$ norm. Another strong convexity parameters we will need is $\mu_L$ defined with respect to the norm $\norm{\z}_L = \sum_{i=1}^d \sqrt{L_i} \abs{\z[i]}$. 
It is shown by~\citet{Nutini_15a} that $\mu_L \ge \mu / (d \bar{L})$.

We collect relevant parameters for CD methods in Table~\ref{t:CD-notation}.
\begin{table*}[t]
\centering
\caption{Notation quick reference for CD methods.}
\label{t:CD-notation}
\begin{tabular}{|c|c|c|}\hline
Notation & Description & Relevant properties for our least squares problems \\ \hline
$L_i$ & smoothness parameter for $i$-th coordinate & $L_i :=  \lambda - \A[ii]$ \\
$\bar{L}$ & average smoothness parameter of all coordinates & $\bar{L} = \lambda - \frac{1}{d} \sum_{i=1}^d \A[ii] = \lambda - \frac{1}{d} \sum_{i=1}^d \rho_d$ \\
$L_{\max}$ & largest smoothness parameter of all coordinates & $L_{\max}:=\max_{1\le i \le d}\; L_i$ \\
$\tilde{L}$ & global smoothness parameter & $\tilde{L} = \lambda - \rho_d, \qquad \bar{L} \le L_{\max} \le  \tilde{L} \le d L_{\max}$ \\ \hline
$\mu$ & strong-convexity parameter in $\ell_2$-norm & $\mu=\lambda - \rho_1$ \\
$\mu_L$ & strong-convexity parameter in $\norm{\cdot}_L$-norm & $\mu_L \ge \mu / (d \bar{L})$ \\
\hline
\end{tabular}
\end{table*}

\subsection{Coordinate selection}
\label{sec:coord-descent-selection-rule}

The choice of coordinate to update is a central research topic for CD methods. 
We now discuss several coordinate selection rules, along with convergence rates, that are most relevant in our setting.

\textbf{Gauss-Southwell-Lipschitz (GSL,~\citealp{Nutini_15a})}: A greedy rule for selecting the coordinate where the coordinate with largest gradient magnitude (relative to the square root of Lipschitz constant) is chosen. The greedy rule tends to work well for high dimensional sparse problems~\citep{Dhillon_11a}.

\textbf{Cyclic~\citep{BeckTetruas03a}}: Given an ordering of the coordinates, cyclic coordinate descent processes each coordinate exactly once according to the ordering in each pass. In the extension ``essentially cyclic'' rule, each coordinate is guaranteed to be chosen at least once in every $\tau \ge d$ updates. This rule is particularly useful when the data can not fit into memory, as we can iteratively load a fraction of $\A$ and update the corresponding coordinates, and still enjoy convergence guarantee.

\textbf{Accelerated randomized coordinate descent methods (ACDM)}: In each step, a coordinate is selected randomly, according to certain (possibly non-uniform) distribution based on the coordinate-wise Lipschitz constant. In accelerated versions of randomized CD~\citep{Nester12a,LeeSidfor13a, LuXiao15a, Allen-Zhu16a}, one can maintain auxiliary sequences of iterates to better approximate the objective function and obtain faster convergence rate, at the cost of (slightly) more involved algorithm. We implement the variant NU-ACDM by~\cite{Allen-Zhu16a}, which samples each coordinate $i$ with probability proportional to $\sqrt{L_i}$ and has the fastest convergence rate to date.

\begin{algorithm}[t]
  \caption{Coordinate descent for solving $f(\x) = \frac{1}{2} \x^\top (\I - \A) \x - \y^\top \x$.}
  \label{alg:gsl}
  \renewcommand{\algorithmicrequire}{\textbf{Input:}}
  \renewcommand{\algorithmicensure}{\textbf{Output:}}
  \begin{algorithmic}
    \REQUIRE Data $\A \in \bbR^{d \times d}$, $\b \in \bbR^{d}$, initialization $\x_0$.
    \STATE Compute gradient $\g \leftarrow \lambda \x - \A \x - \y$
    \FOR{$t=1,2,\dots,\tau$}
    \STATE Select a coordinate using one of the rules
    \begin{align*}
      \text{GSL:}& \qquad j \leftarrow \argmax_{1\le i\le d}\; \frac{\abs{\g[i]}}{\sqrt{\lambda - \A[ii]}} \\
      \text{Cyclic:}& \qquad j \leftarrow \mod(t,d)+1  \\
      \text{Random:}& \qquad j \leftarrow \text{random index} \; \in [1,d] 
    \end{align*}
    \STATE Compute update: 
    \hspace*{1em} $\delta \leftarrow - \frac{\g[j]}{\lambda - \A[jj]}$
    \STATE Update coordinate:\hspace*{1em} $\x_{t} \leftarrow \x_{t-1} + \delta \cdot \e_j$
    \STATE Update gradient:
    \begin{align*}
      \g \leftarrow \g - \delta \cdot \A[:j],\qquad \g[j] \leftarrow \g[j] + \lambda \delta
    \end{align*}
    \ENDFOR  
    \ENSURE $\x_{\tau}$ is the approximate solution.
  \end{algorithmic}
\end{algorithm}

We provide the pseudocode of coordinate descent of the above rules in Algorithm~\ref{alg:gsl}.\footnote{The pseudo code for acclerated randomized CD is more involved and we refer the readers to~\cite{Allen-Zhu16a}.}  Note that the stepsize for the chosen coordinate $j$ is $\frac{1}{L_j}$, the inverse Lipschitz constant; for least squares problems where $f(\x)$ is quadratic in each dimension, this stepsize exactly minimizes the function over $\x[j]$ given the rest coordinates. In some sense, the GSL rule is the ``optimal myopic coordinate update''~\citep{Nutini_15a}.

\paragraph{Connection to the greedy selection rule of~\citet{Lei_16a}} Given the current estimate $\x$, the coordinate-wise power method of~\citet{Lei_16a} selects the following coordinate to be updated
\begin{align*}
\argmax_{i}\; \abs{\c[i]} \quad \text{where}\; \c = \frac{\A \x}{\x^\top \A \x} - \x,
\end{align*}
\ie, the coordinate that would be updated the most if a full power iteration were performed. Their selection rule is equivalent to choosing the largest element of $((\x^\top \A \x) \I - \A) \x$, where $\x^\top \A \x$ is the current (lower) estimate of the eigenvalue $\rho_1$.  On the other hand, the greedy Gauss-Southwell rule for our method chooses the largest element of the gradient $(\lambda \I - \A) \x - \y$, where $\y$ is an earlier estimate of the eigenvector and $\lambda$ is an upper estimate of $\rho_1$.

\subsection{Convergence properties}
\label{sec:coord-descent-convergence-rate}

We quote the convergence of CD from the literature.
\begin{lem} [Iteration complexity of CD methods] \label{lem:CD}
The numbers of coordinate updates for Algorithm~\ref{alg:gsl} to achieve $f (\x_\tau) - f_*  \le \epsilon^\prime$ where $f_* = \min_{\x} f(\x)$, using different coordinate selection rules, are
  \begin{gather*}
\text{GSL:}\quad \calO \left(\frac{1}{\mu_L} \cdot \log \frac{f(\x_0) - f_*}{\epsilon^\prime} \right), \\ 
\text{Cyclic:}\quad \calO \left(\frac{d L_{\max} (1 + d \tilde{L}^2 / L_{\max}^2)}{\mu} \cdot \log \frac{f(\x_0) - f_*}{\epsilon^\prime} \right), \\
\text{NU-ACDM:}\quad \calO \left(\frac{\sum_{i=1}^d \sqrt{L_i}}{\sqrt{\mu}} \cdot \log \frac{f(\x_0) - f_*}{\epsilon^\prime} \right).
  \end{gather*}
\end{lem}

%
%

  Lemma~\ref{lem:CD} implies that the coordinate descent methods converges linearly for our strongly convex objective: the suboptimality decreases at different geometric rates for each method. Most remarkable is the time complexity of ACDM which has dependence on $\frac{1}{\sqrt{\mu}}$: as we mentioned earlier, the strong convexity parameter $\mu$ for our least squares problems (in the accurate regime) are of the order $\Delta$, thus instantiating the shift-and-invert power method with ACDM leads to a total time complexity of which depends on $\frac{1}{\sqrt{\Delta}}$. In comparison, the total time complexity of standard power method depends on $\frac{1}{\Delta}$). Thus we expect our algorithm to be much faster when the eigengap is small.

It is also illuminating to compare ACDM with full gradient descent methods. The iteration complexity to achieve the same accuracy is $\calO \left( \frac{\sqrt{\tilde{L}}}{\sqrt{\mu}} \right) $ for accelerated gradient descent (AGD,~\citealp{Nester04a}). However, each full gradient calculation cost $\calO (d^2)$ (not taking into account the sparsity) and thus the total time complexity for AGD is $\calO \left( \frac{d^2 \sqrt{\tilde{L}}}{\sqrt{\mu}} \right)$. In contrast, each update of ACDM costs $\calO (d)$, giving a total time complexity of $\calO \left( \frac{d \sum_{i=1}^d \sqrt{L_i}}{\sqrt{\mu}} \right)$. In view of~\eqref{e:relation-between-smoothness}, we have 
\begin{align*}
\frac{d \sum_{i=1}^d \sqrt{L_i}}{\sqrt{\mu}} \le 
\frac{d^2  \sqrt{L_{\max}}}{\sqrt{\mu}} \le 
\frac{d^2 \sqrt{\tilde{L}}}{\sqrt{\mu}},
\end{align*}
and thus ACDM is typically more efficient than AGD (and in the extreme case of all but one $L_i$ being $0$, the rate of ACDM is $d$ times better). 
It is also observed empirically that ACDM outperforms AGD and the celebrated conjugate gradient method for solving least squares problems~\citep{LeeSidfor13a,Allen-Zhu16a}. 

Although the convergence rates of GSL and cyclic rules are inferior than that of ACDM, we often observe competitive empirical performance from them. They are easier to implement and can be the method of choice in certain scenariors (\eg, cyclic coordinate descent may be employed when the data can not fit in memory). Finally, we remark that our general scheme and conclusion carries over to the case of block and parallel coordinate descent methods~\citep{Bradley_11a,RichtarTakac15a}, where more than one coordinate are updated in each step.

\section{Putting it all together} 

Our proposed approach consists of instantiating Algorithm 1 with the  different least squares optimizers (LSO) presented in Algorithm 2. Our analysis of the total time complexity thus combines the iteration complexity of shift-and-invert and the time complexity of the CD methods. First, by Lemma~\ref{lem:shift-and-invert-repeat-until} the time complexity of Phase I is not dominant because it is independent of the final error $\epsilon$. Second, by Lemma~\ref{lem:shift-and-invert-for-loop} we need to solve $\log \frac{d}{\epsilon}$ subproblem, each with constant ratio between initial and final suboptimality. Third, according to Lemma~\ref{lem:CD} and noting $\lambda = \rho_1 + \calO(1) \cdot \Delta$, the number of coordinate updates for solving each subproblem is $\calO \left( \frac{1}{\mu_L} \right) = \calO \left( \frac{d (\rho_1 - \frac{1}{d} \sum_{i=1}^d \rho_d)}{\Delta}\right)$ for GSL, and $\calO \left( \frac{\sum_{i=1}^d \sqrt{L_i}}{\sqrt{\mu}} \right) = \calO \left(\frac{ \frac{1}{d} \sum_{i=1}^d \sqrt{ \rho_1 -  \A_{ii} }}{\sqrt{\Delta}} \right)$ for ACDM. 
Since each coordinate update costs $\calO (d)$ time, we obtain the time complexity of our algorithms given in Table~\ref{t:complexity}.

%% file: algorithm3.tex
\section{Extensions}
\label{sec:algorithm-extension}

Our algorithms can be easily extended to several other settings. Although our analysis have assumed that $\A$ is positive semidefinite, given an estimate of the maximum eigenvalue, the shift-and-invert algorithm can always convert the problem into a series of convex least squares problems, regardless of whether $\A$ is positive semidefinite, and so CD methods can be applied with convergence guarantee. To extract the tailing eigenvector of $\A$, one can equivalently extract the top eigenvector of $-\A$. Futhermore, extracting the top singular vectors of a non-symmetric matrix $\A$ is equivalent to extracting the top eigenvector of $\left[ \begin{array}{cc} \0 & \A^\top \\ \A & \0 \end{array} \right]$.

To extend our algorithms to extracting multiple top eigenvectors setting, we can use the ``peeling'' procedure: we iteratively extract one eigenvector at a time, removing the already extracted component from the data matrix before extracting the new direction. For example, to extract the second eigenvector $\p_2$, we can equivalently extract the top eigenvector of $\A^\prime=(\I - \p_1 \p_1^\top) \A (\I - \p_1 \p_1^\top)$. And note that we do not need to explicitly compute and store $\A^\prime$ which may be less sparse than $\A$; all we need in CD methods are columns of $\A^\prime$ which can be evaluated by $\calO (d)$ vector operations, \eg, $\A^\prime[:j]= (\I - \p_1 \p_1^\top) \cdot \left(\A[:,j] - \p_1[j] \cdot (\A \p_1) \right)$ by storing  $\A \p_1 \in \bbR^d$. We refer the readers to~\citet{Allen-ZhuLi16a} for a careful analysis of the accuracy in solving each direction and the runtime.


%% file: experiment.tex
\section{Experiments}
\label{sec:expts}

We now present experimental results to demonstrate the efficiency of our algorithms, denoted as SI (shift-and-invert) + least squares solvers. Besides CD methods, we include AGD as a solver since SI+AGD gives the same time complexity as Lanczos.
We compare our algorithms with the coordinate-wise power method (CPM) and SGCD of~\citet{Lei_16a}, and also the standard power method.

\subsection{Synthetic datasets}
\label{sec:expts-synthetic}

We first generate synthetic datasets to validate the fast convergence of our algorithms. Our test matrix has the form $\A = \U \SS \U^{\top}$, where $\U$ is a random rotation matrix and $\SS = {\rm diag}(1,1-\Delta,1-2\Delta,\ldots)$. The parameter $\Delta$ (which is also the eigengap of $\A$) controls the decay of the spectrum of $\A$. We use $4$ passes of coordinate updates (each pass has $d$ coordinate updates) for SI-Cyclic, SI-GSL and SI-ACDM, and $4$ accelerated gradient updates for SI-AGD, to approximately solve the least squares problems~\eqref{e:lsq}, and set $\tilde{\Delta} = 0.0001$ in Algorithm 1.

We test several settings of dimension $d$ and $\delta$, and the results are summarized in Figure~\ref{fig:simulated_experiments}. We observe that in most cases, CPM/SGCD indeed converge faster than the standard power method, justifying the intuition that some coodinates are more important than others. Furthermore, our algorithm significantly improve over CPM/SGCD for accurate estimate of the top eigenvector, especially when the eigengap $\Delta$ is small, validating our convergence analysis which has improved dependence on the gap.

\begin{figure*}[t]
\centering
\psfrag{Objective Suboptimality}[][]{Suboptimality}
\psfrag{Number of passes over data}[t][]{\# coordinate passes}
\begin{tabular}{@{}c@{\hspace{0.02\linewidth}}c@{\hspace{0\linewidth}}c@{\hspace{0\linewidth}}c@{\hspace{0\linewidth}}c@{}}
& $\Delta=5\times 10^{-3}$ 
& $\Delta=5\times 10^{-4}$
& $\Delta=5\times 10^{-5}$
& $\Delta=5\times 10^{-6}$ \\[-.5ex]
\rotatebox{90}{\hspace{2.5em}$d=500$} & 
\includegraphics[width=0.24\linewidth]{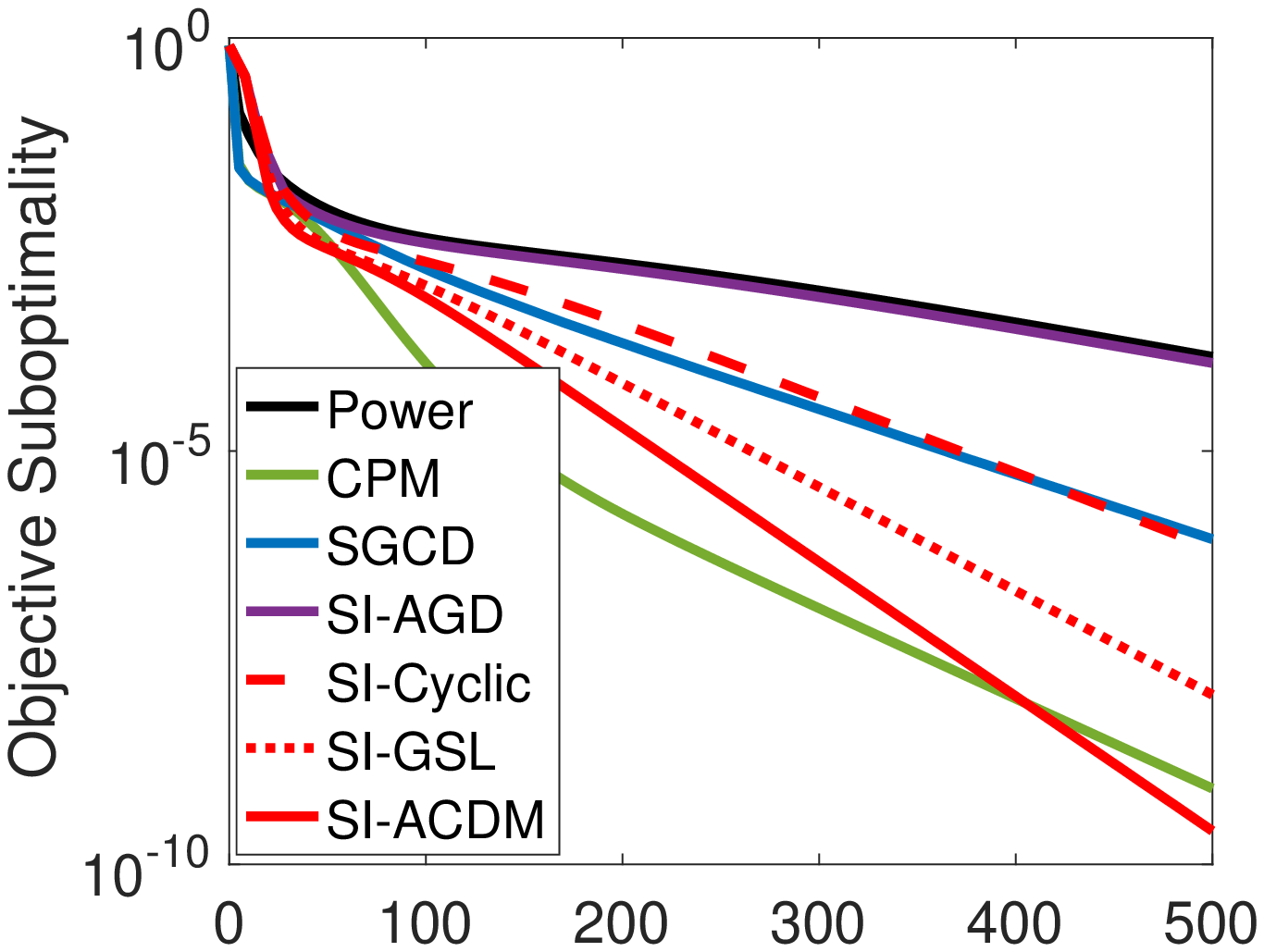}&
\includegraphics[width=0.24\linewidth]{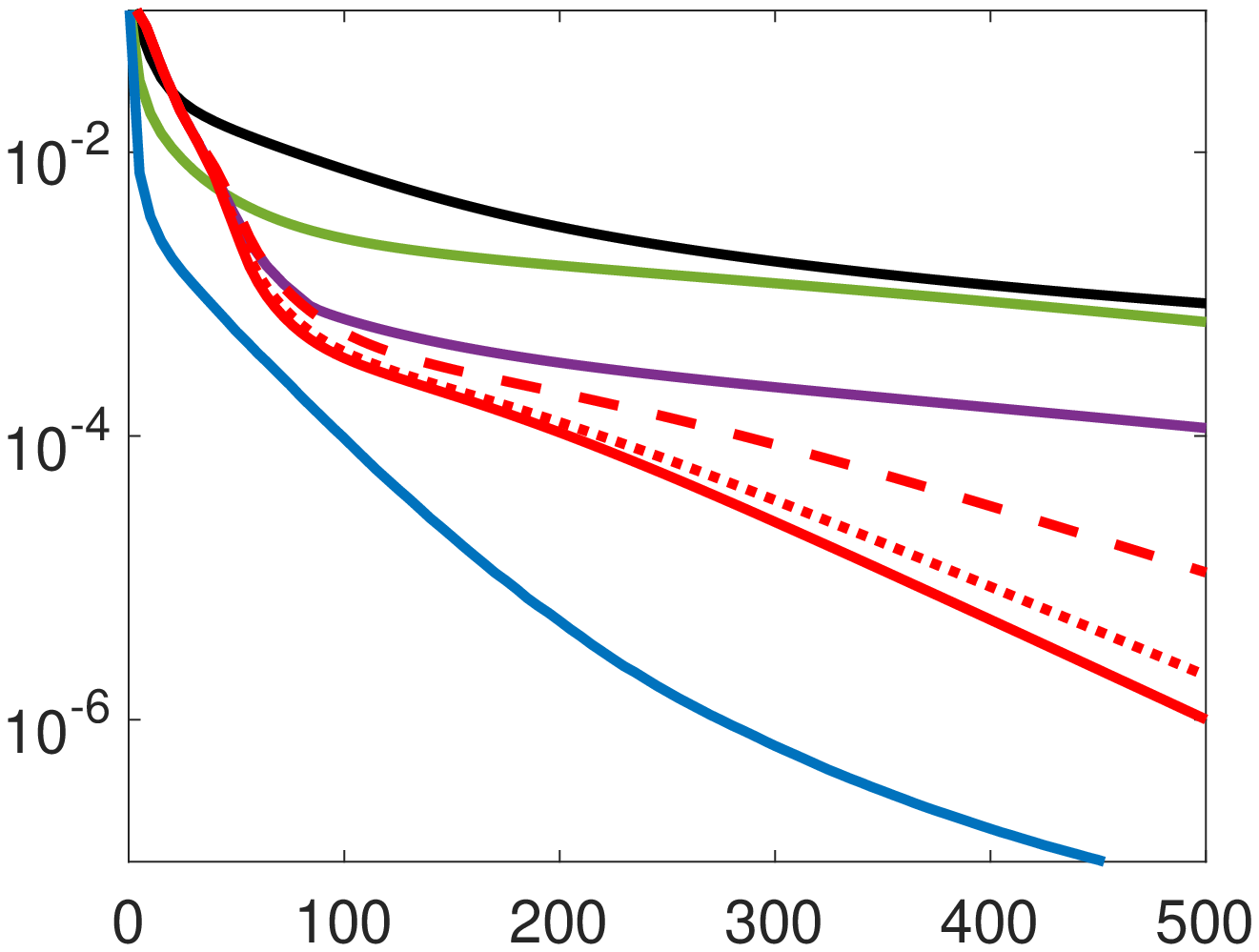}&
\includegraphics[width=0.24\linewidth]{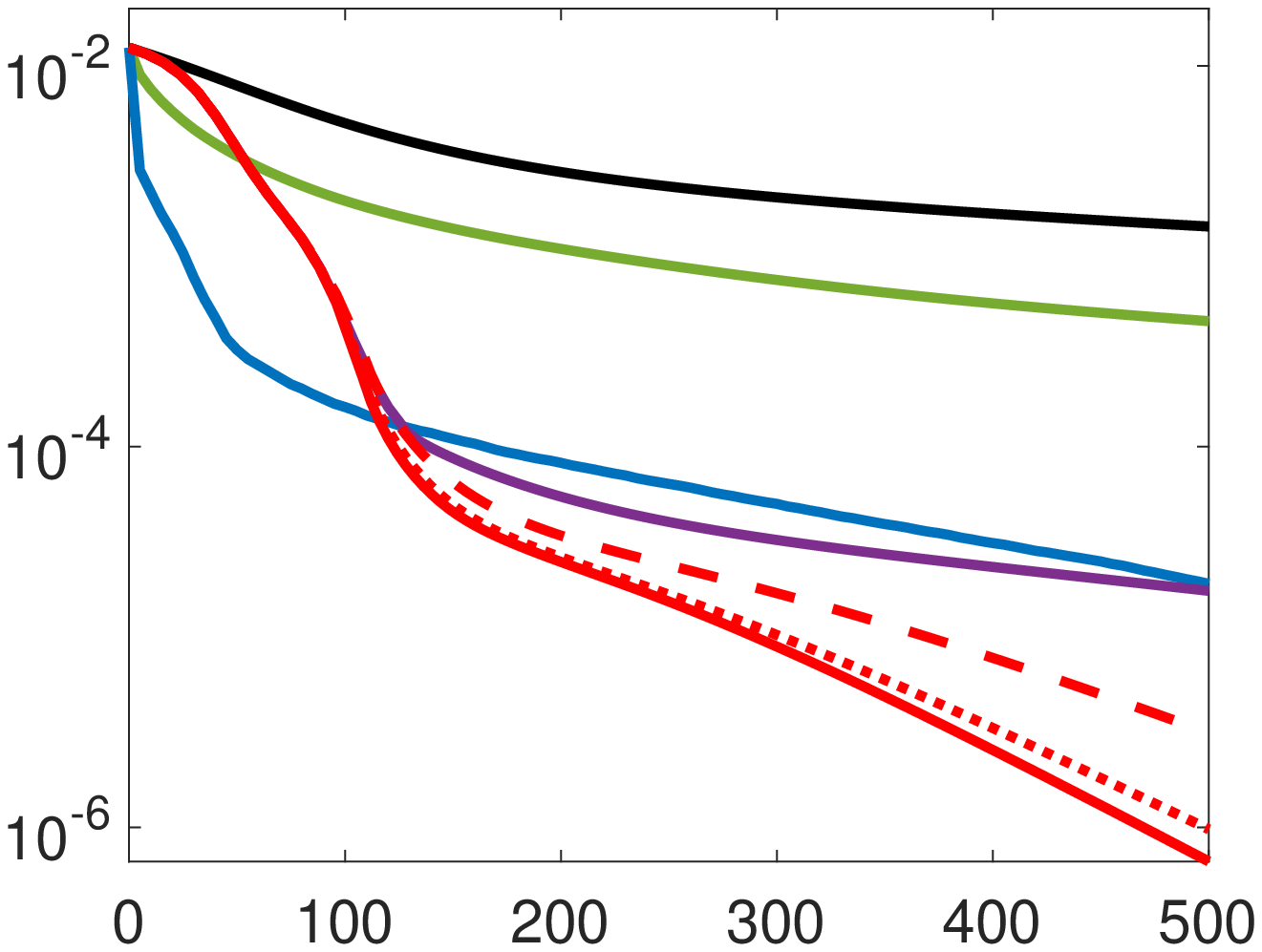}&
\includegraphics[width=0.24\linewidth]{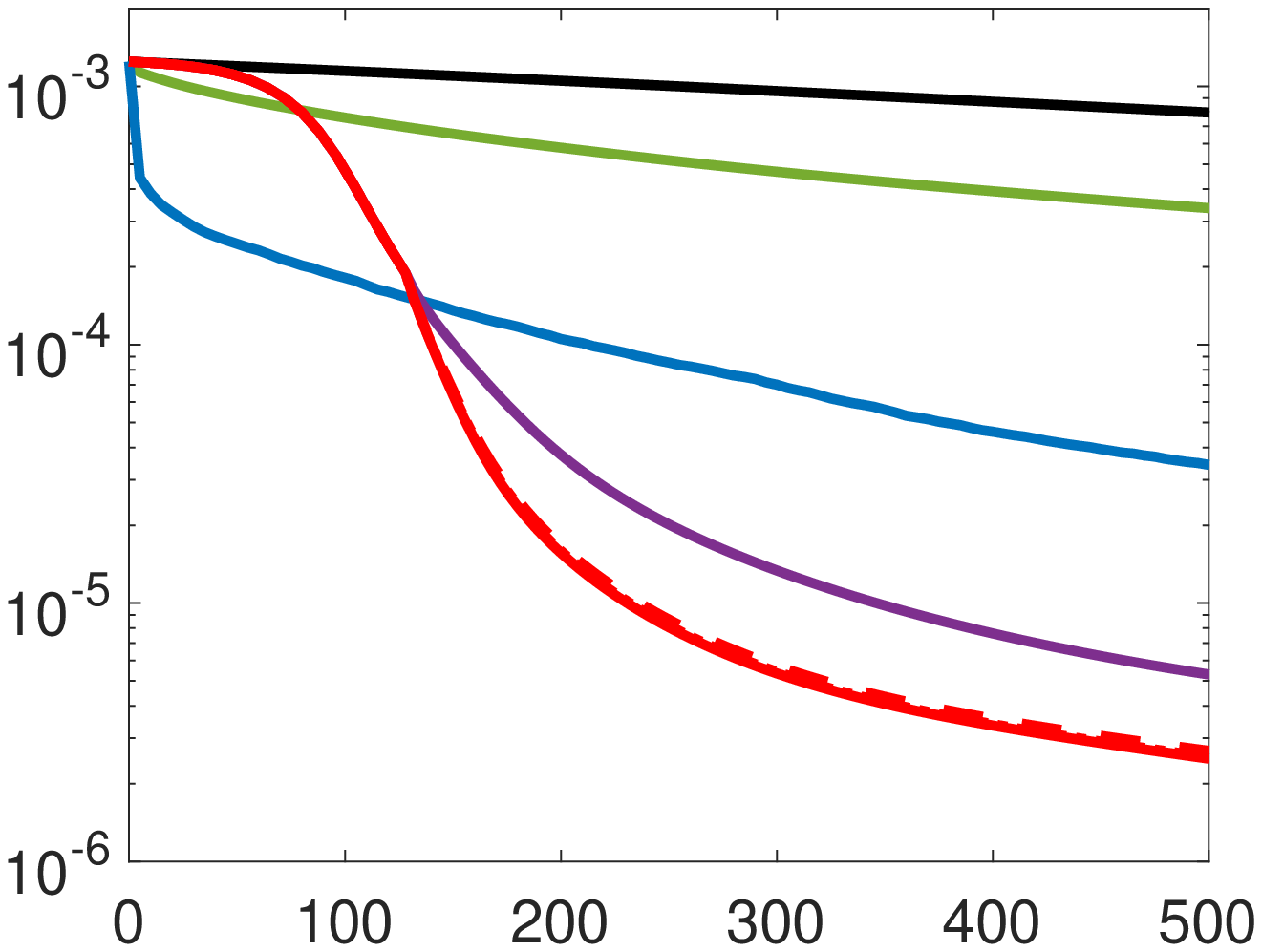}\\[-1ex]
\rotatebox{90}{\hspace{2.5em}$d=2000$} & 
\includegraphics[width=0.24\linewidth]{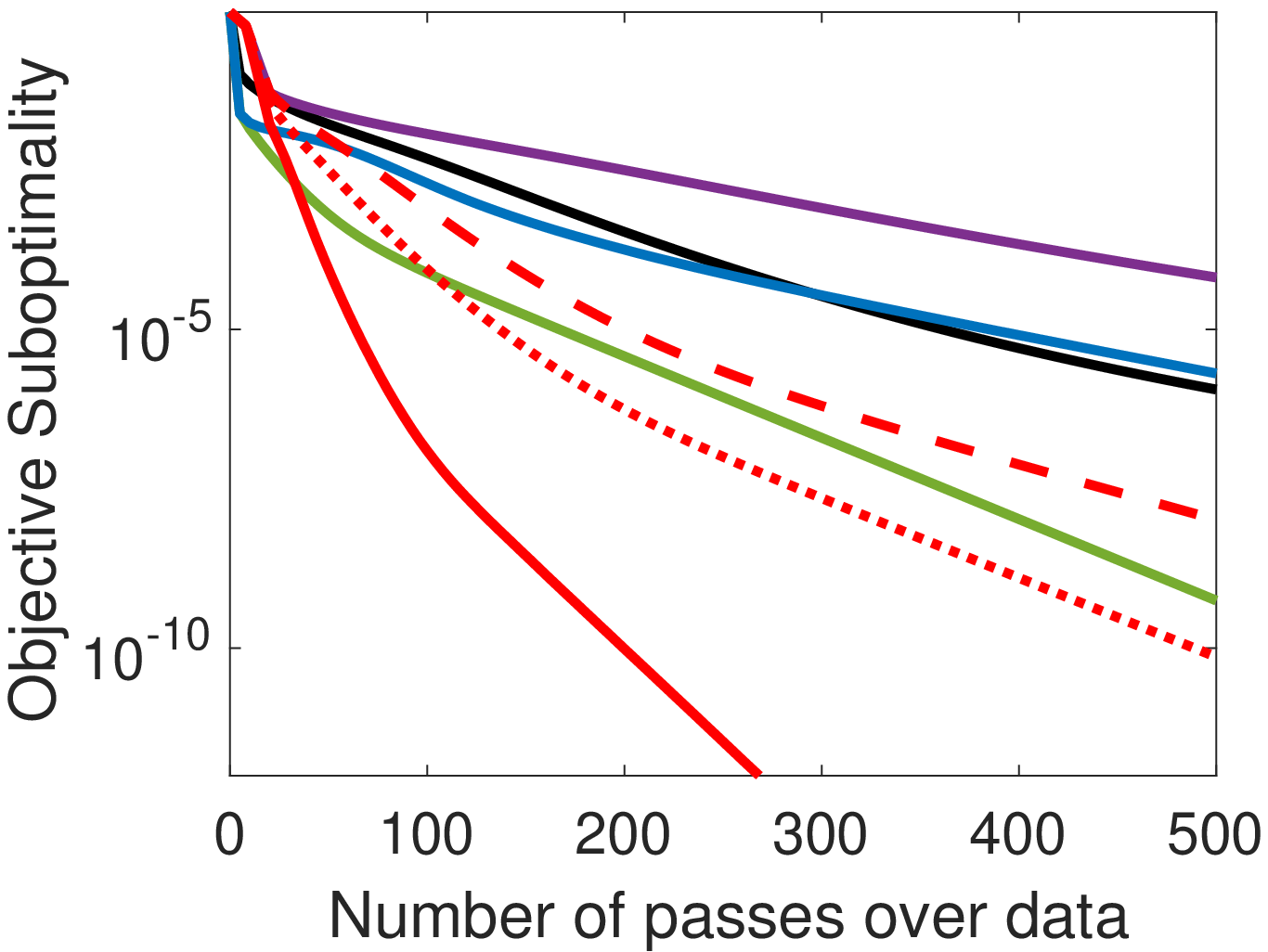}&
\includegraphics[width=0.24\linewidth]{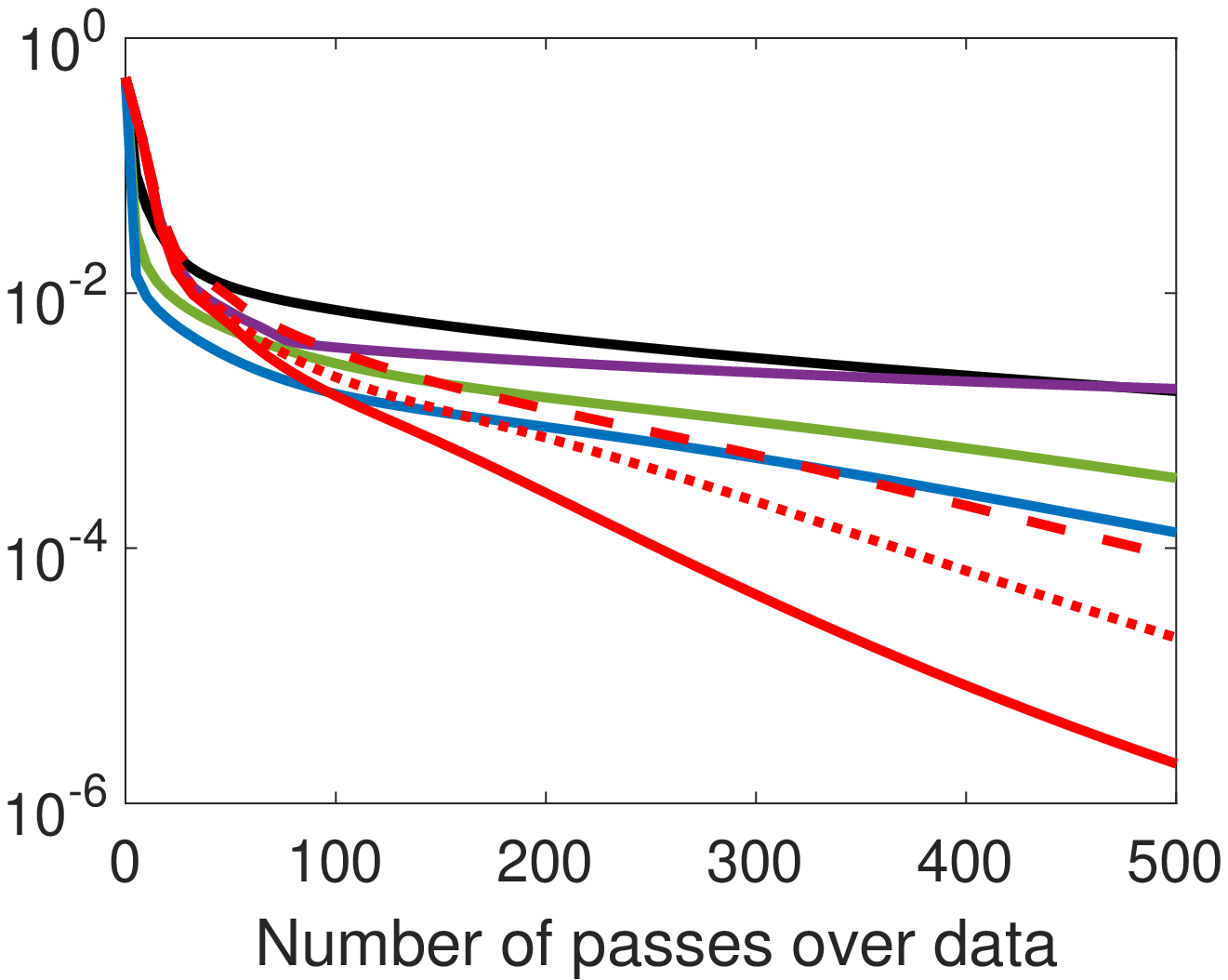}&
\includegraphics[width=0.24\linewidth]{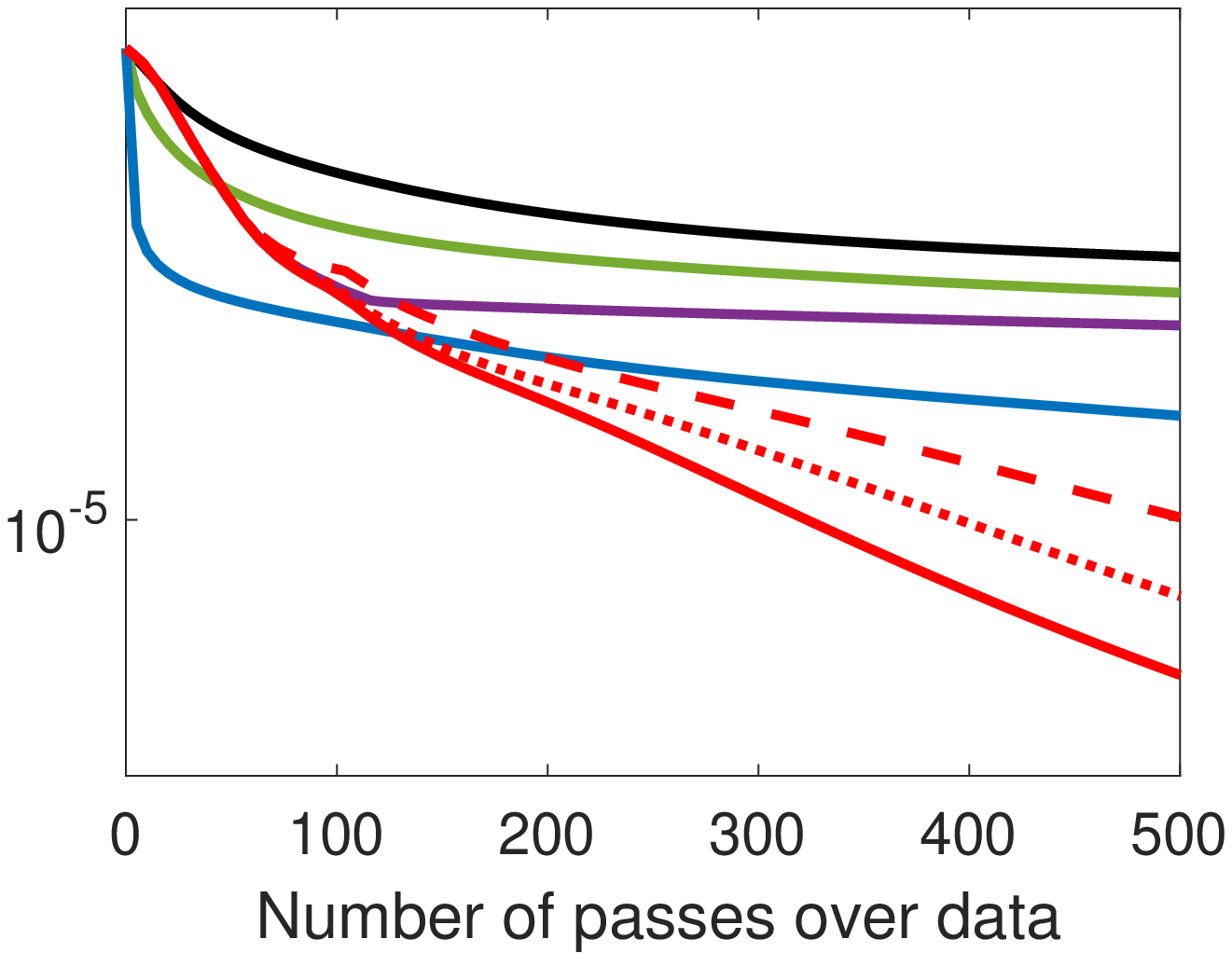}&
\includegraphics[width=0.24\linewidth]{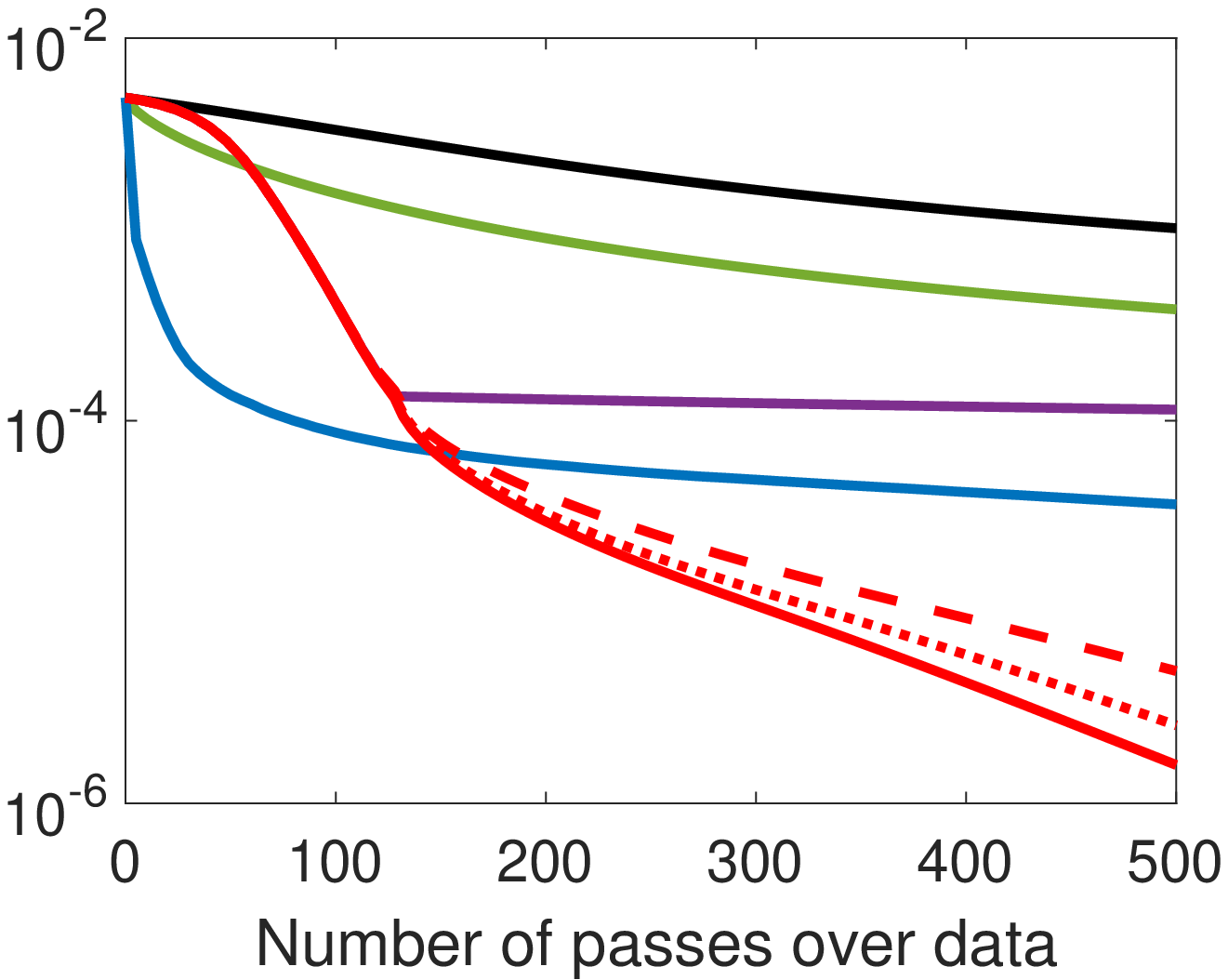}
\end{tabular}
\caption{Comparison of the convergence behavior ($\rho_1 - \w_t^\top \A \w_t$  vs. number of coordinate passes) of various algorithms for computing the leading eigenvector on synthetic datasets.}
\label{fig:simulated_experiments}
\end{figure*}

\subsection{Real-world datasets}
\label{sec:expts-real}

A major source of large-scale sparse eigenvalue problem in machine learning comes from graph-related applications~\citep{ShiMalik00a,Ng_02a,Fowlkes_04a}. 
We now examine the efficiency of the proposed algorithms on several large-scale networks datasets with up to a few million nodes\footnote{\url{http://snap.stanford.edu/data/index.html}}, whose the statistics are summarized in Table~\ref{tab:data} (see Appendix~\ref{sec:statistics-realdata}). 
Let $\W$ be the (binary) adjacency matrix of each network, our goal is to compute the top eigenvectors for the corresponding normalized graph Laplacian matrix~\citep{Luxbur07a} defined as $\A = \D^{-1/2} (\D - \W) \D^{-1/2}$, where $\D$ is a diagonal matrix containing the degrees of each node on the diagonal. This task can be extended to computing a low-rank approximation of the normalized graph Laplacian~\citep{GittenMahoney13a}.~\footnote{We have also experimented with the original task of ~\citet{Lei_16a}. For their task, our algorithm do not achieve significant improvement over CPM/SGCD because the gap of the unnormalized adjacency matrix $\W$ is quite large.} 
We discuss implementation details in Appendix~\ref{sec:statistics-realdata}.

\begin{figure*}[b]
\centering
\psfrag{Objective Suboptimality}[][]{Suboptimality}
\psfrag{Time(s)}[][]{\scriptsize Time (seconds)}
\begin{tabular}{@{}c@{\hspace{0\linewidth}}c@{\hspace{0\linewidth}}c@{\hspace{0\linewidth}}c@{}}
amazon0601 & com-Youtube & email-Enron & email-EuAll \\
\includegraphics[width=0.25\linewidth]{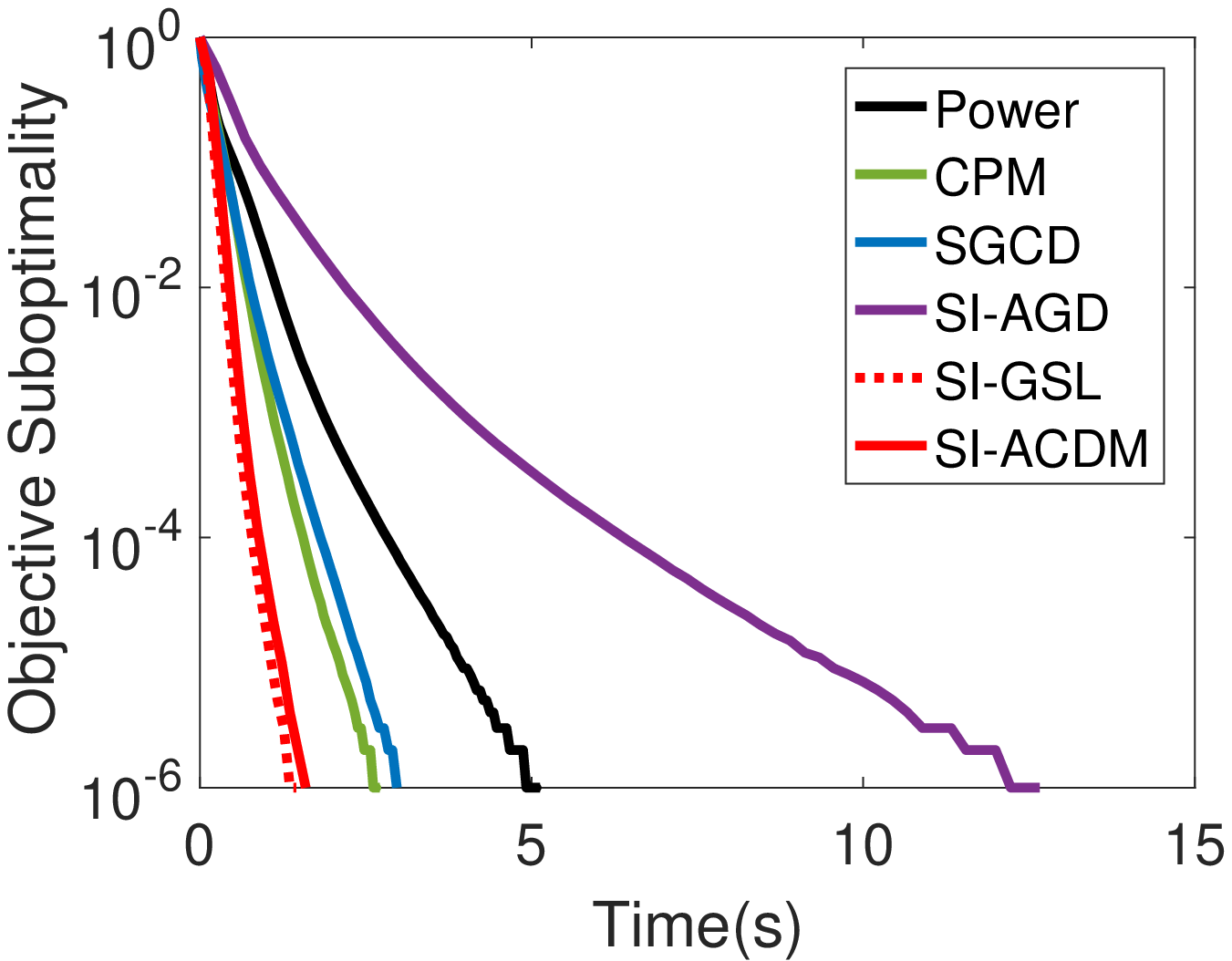} & 
\includegraphics[width=0.25\linewidth]{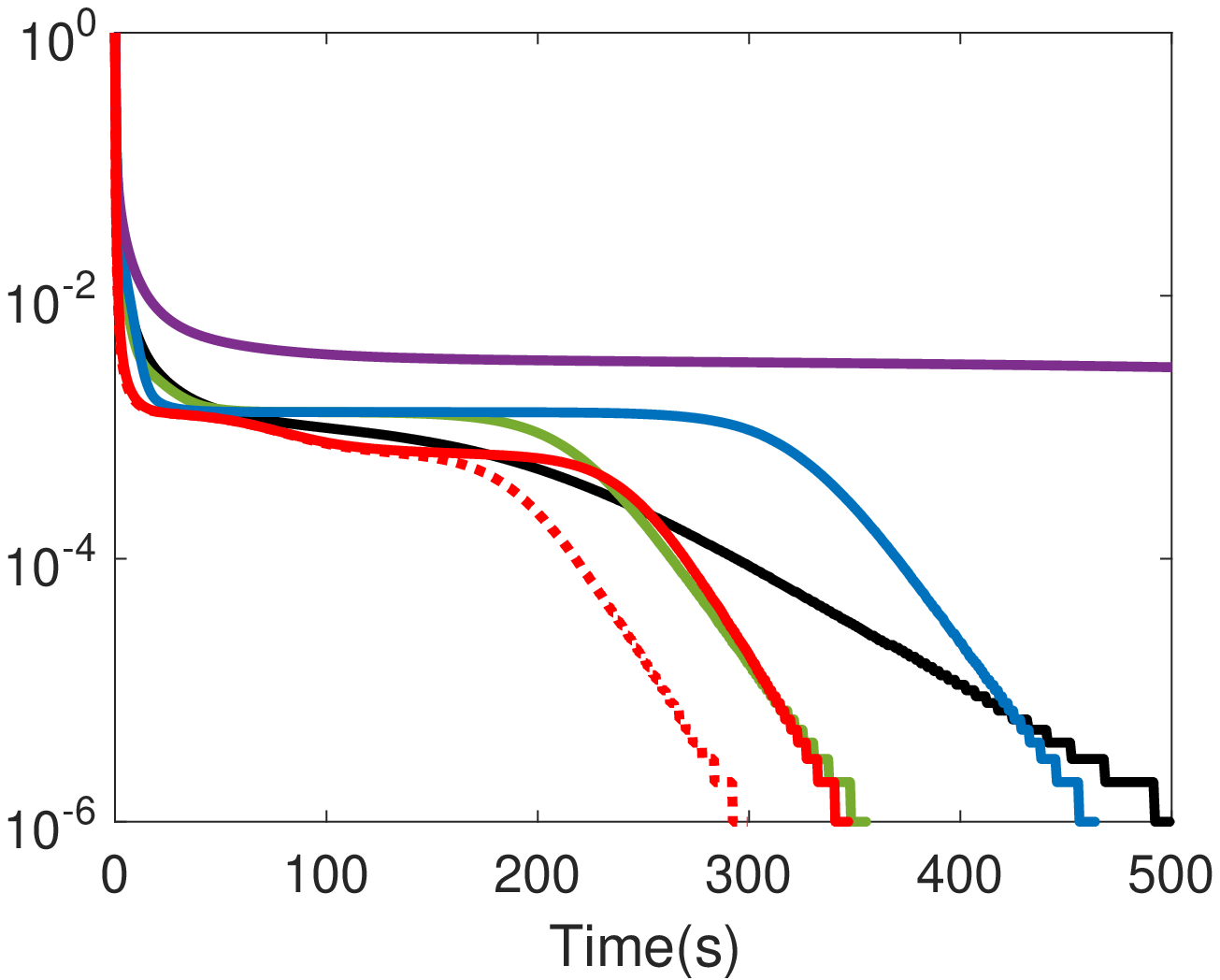} &
\includegraphics[width=0.25\linewidth]{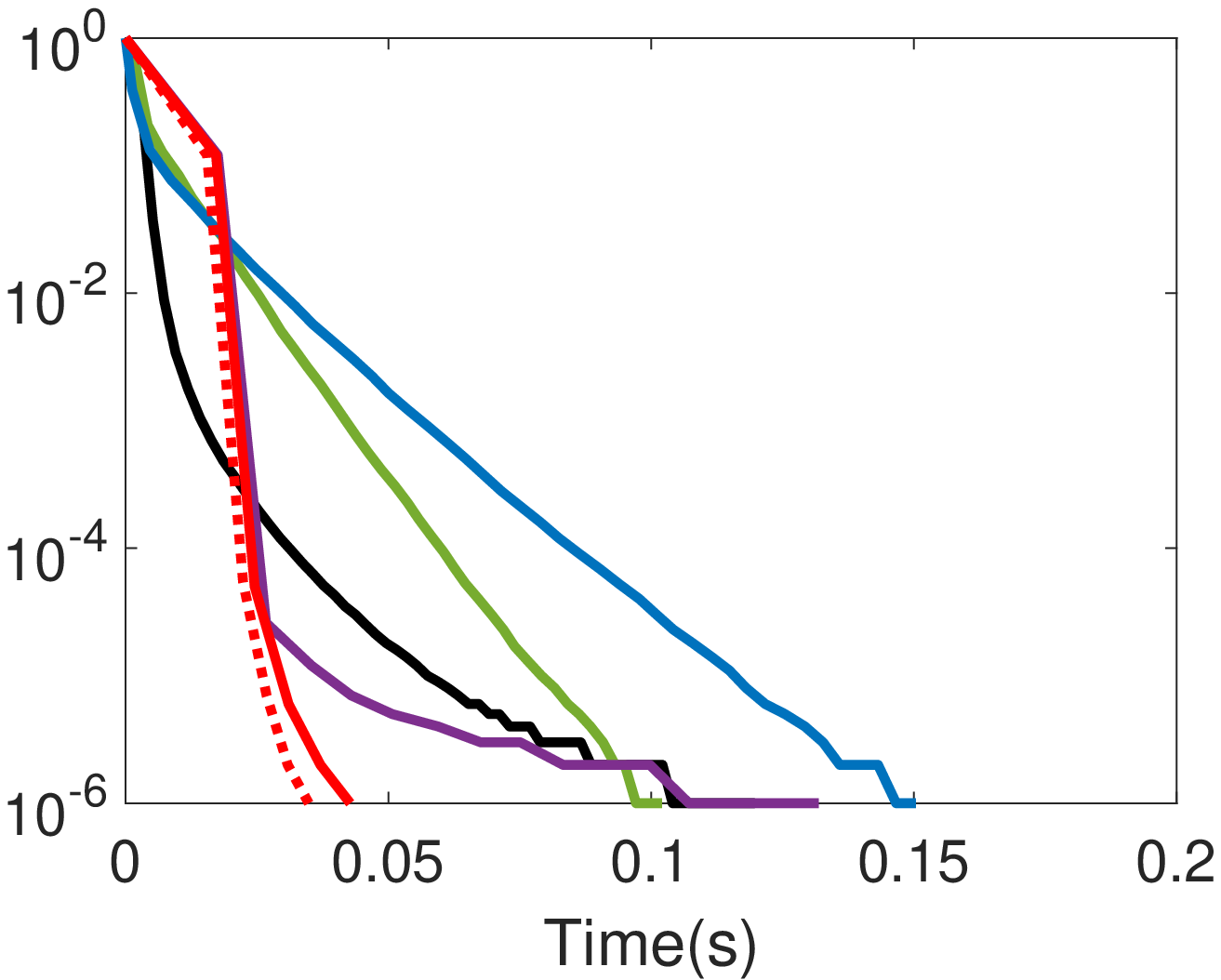} & 
\includegraphics[width=0.25\linewidth]{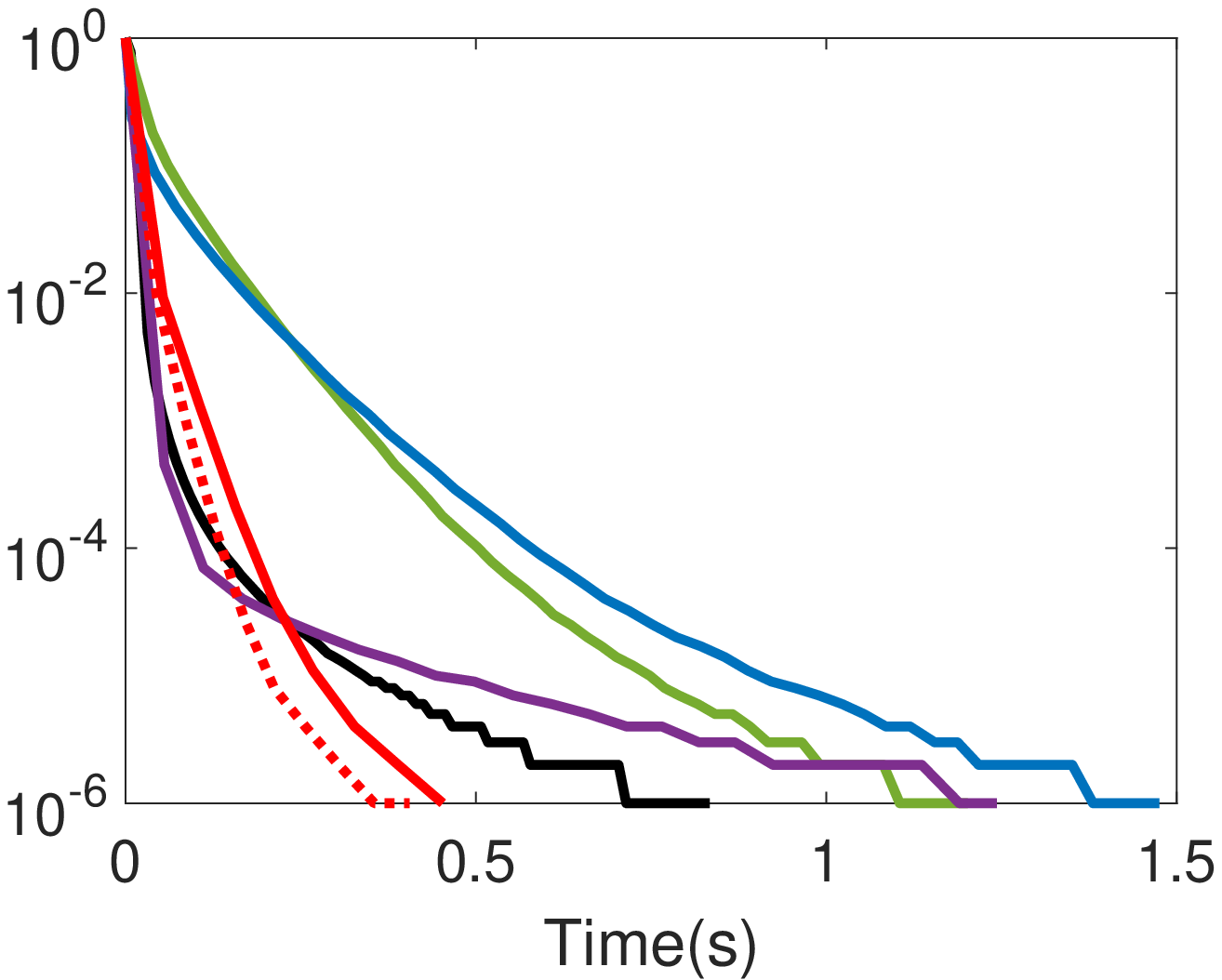} \\[1ex]
p2p-Gnutella31 & roadNet-CA & roadNet-PA & soc-Epinions1 \\
\includegraphics[width=0.25\linewidth]{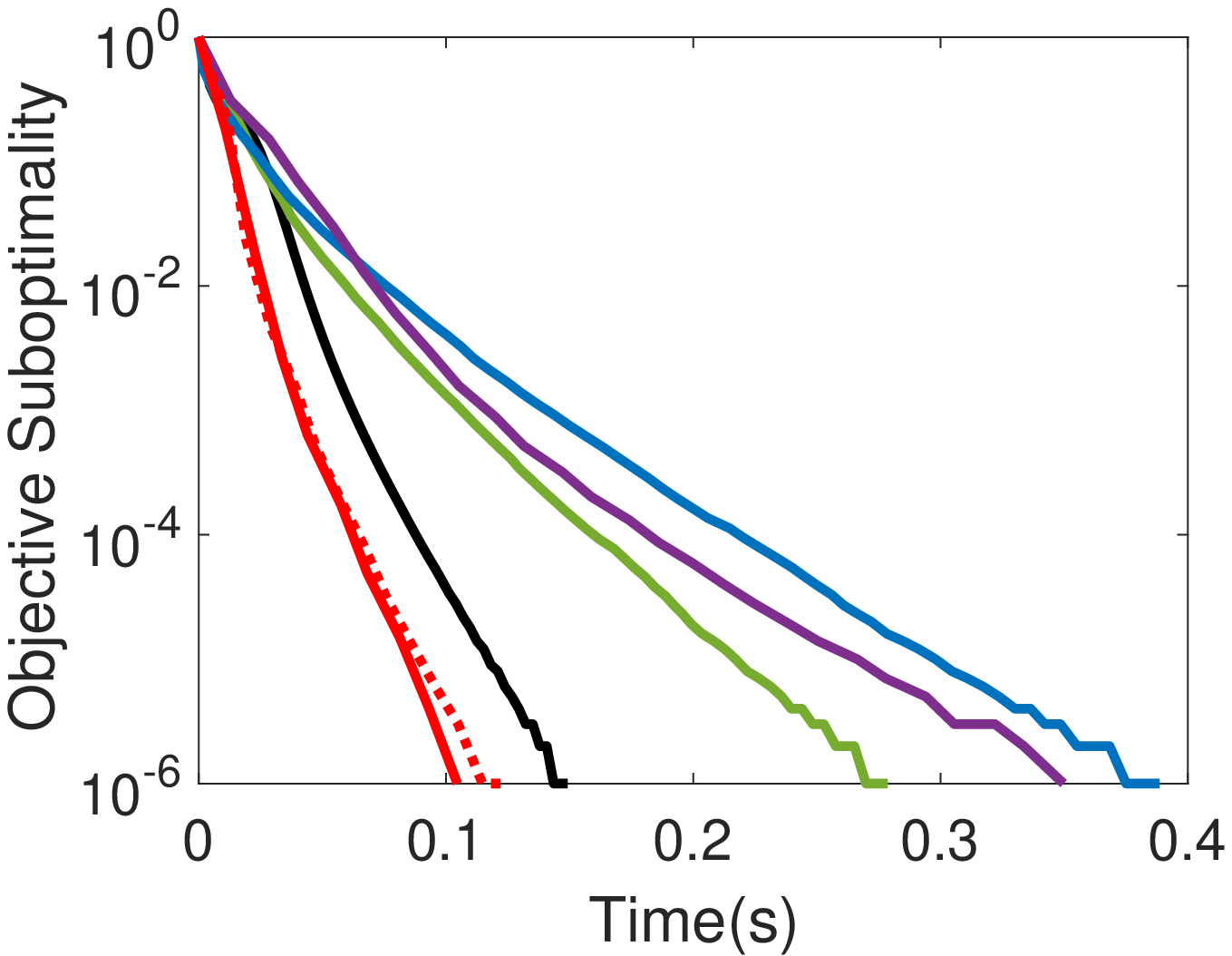} & 
\includegraphics[width=0.25\linewidth]{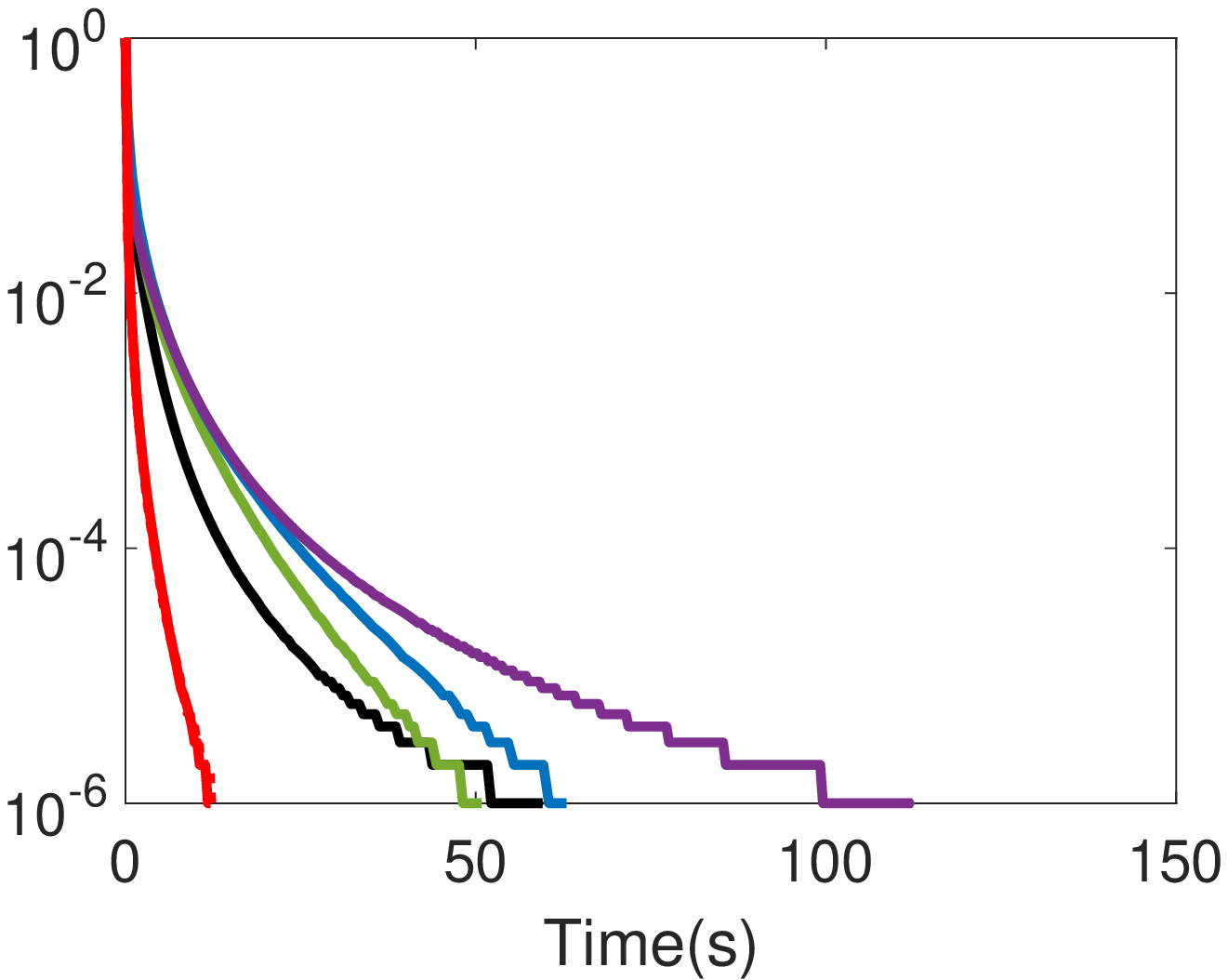} &
\includegraphics[width=0.25\linewidth]{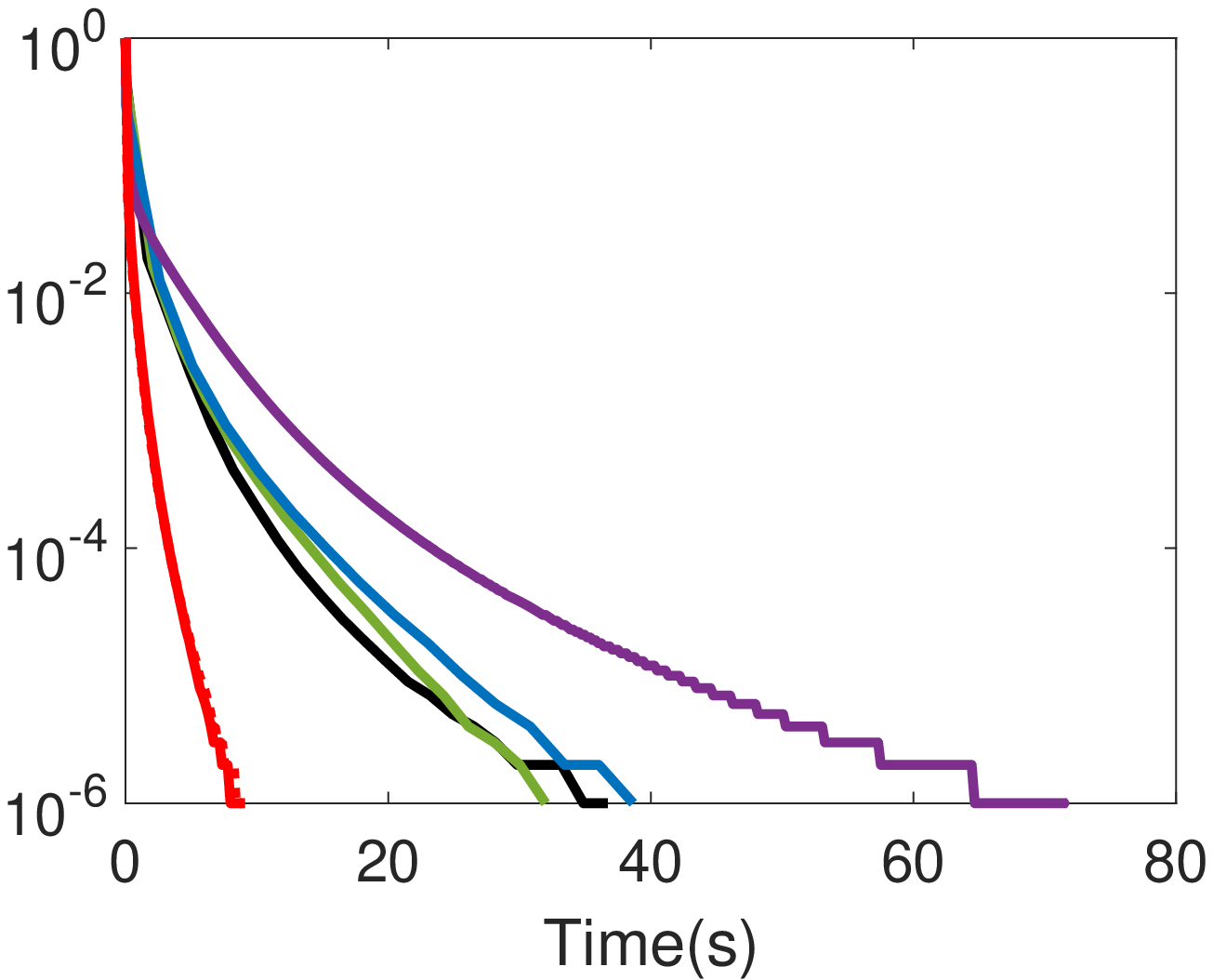} & 
\includegraphics[width=0.25\linewidth]{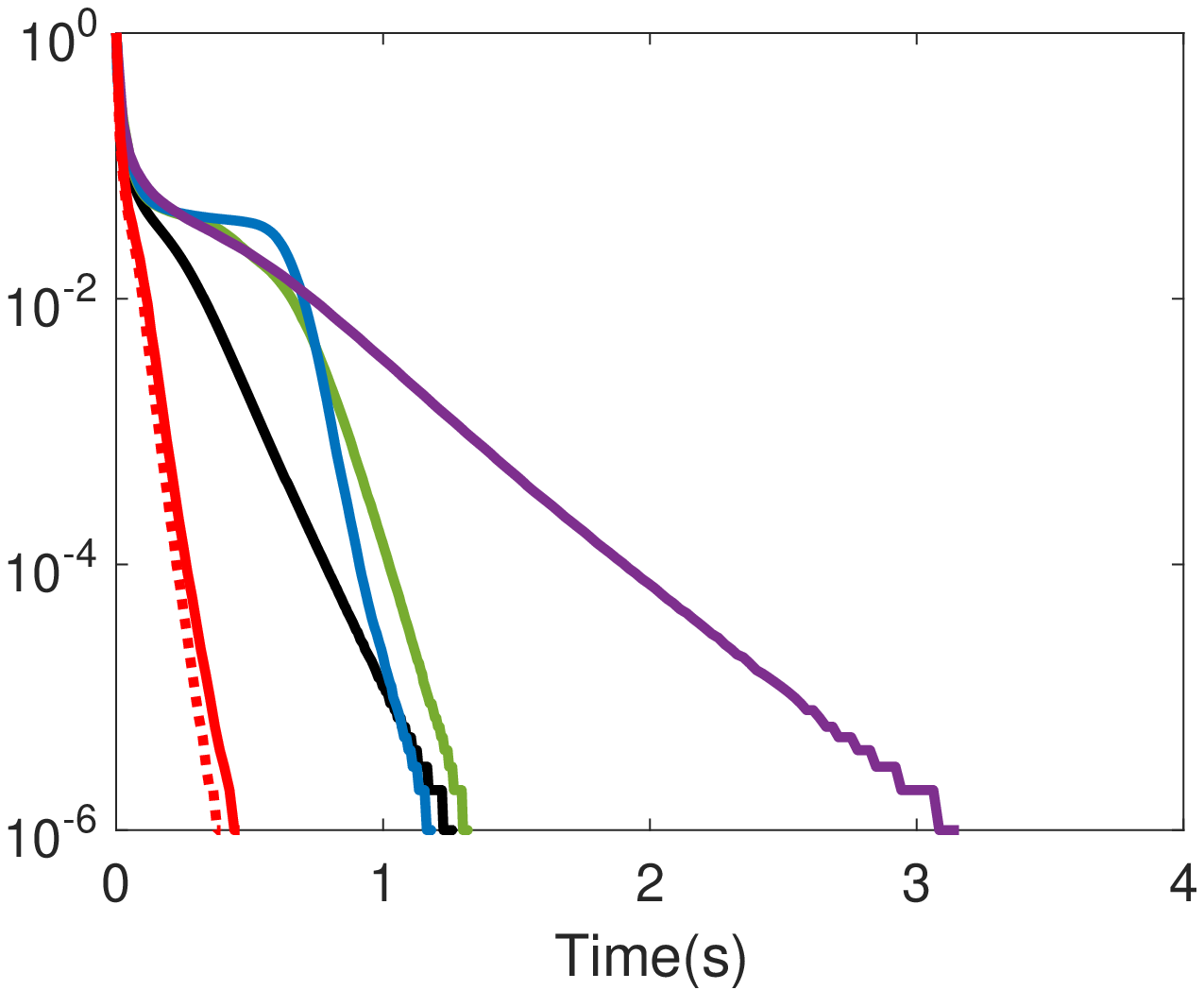} \\[1ex]
web-BerkStan & web-Google & web-NotreDame & wiki-Talk \\
\includegraphics[width=0.25\linewidth]{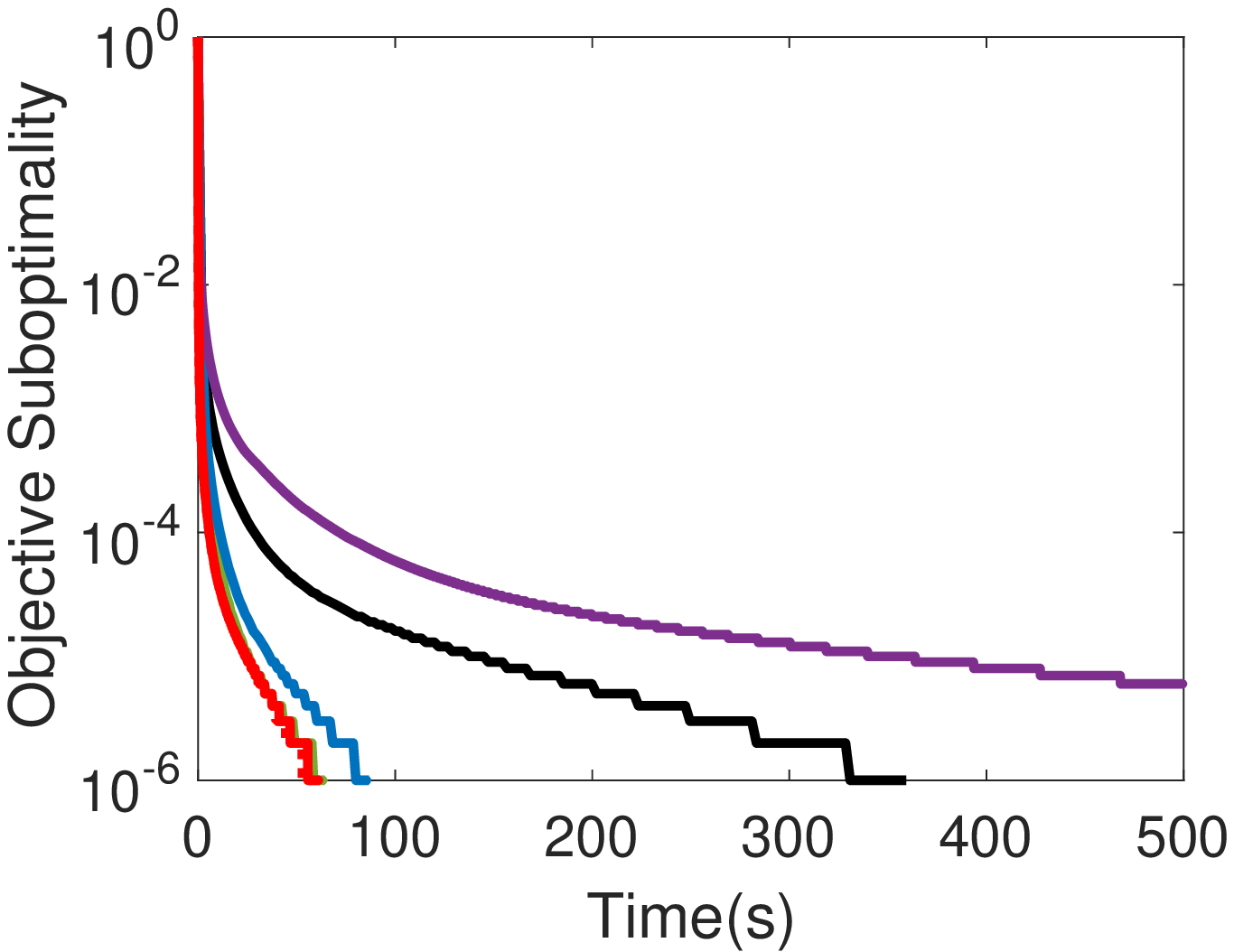} & 
\includegraphics[width=0.25\linewidth]{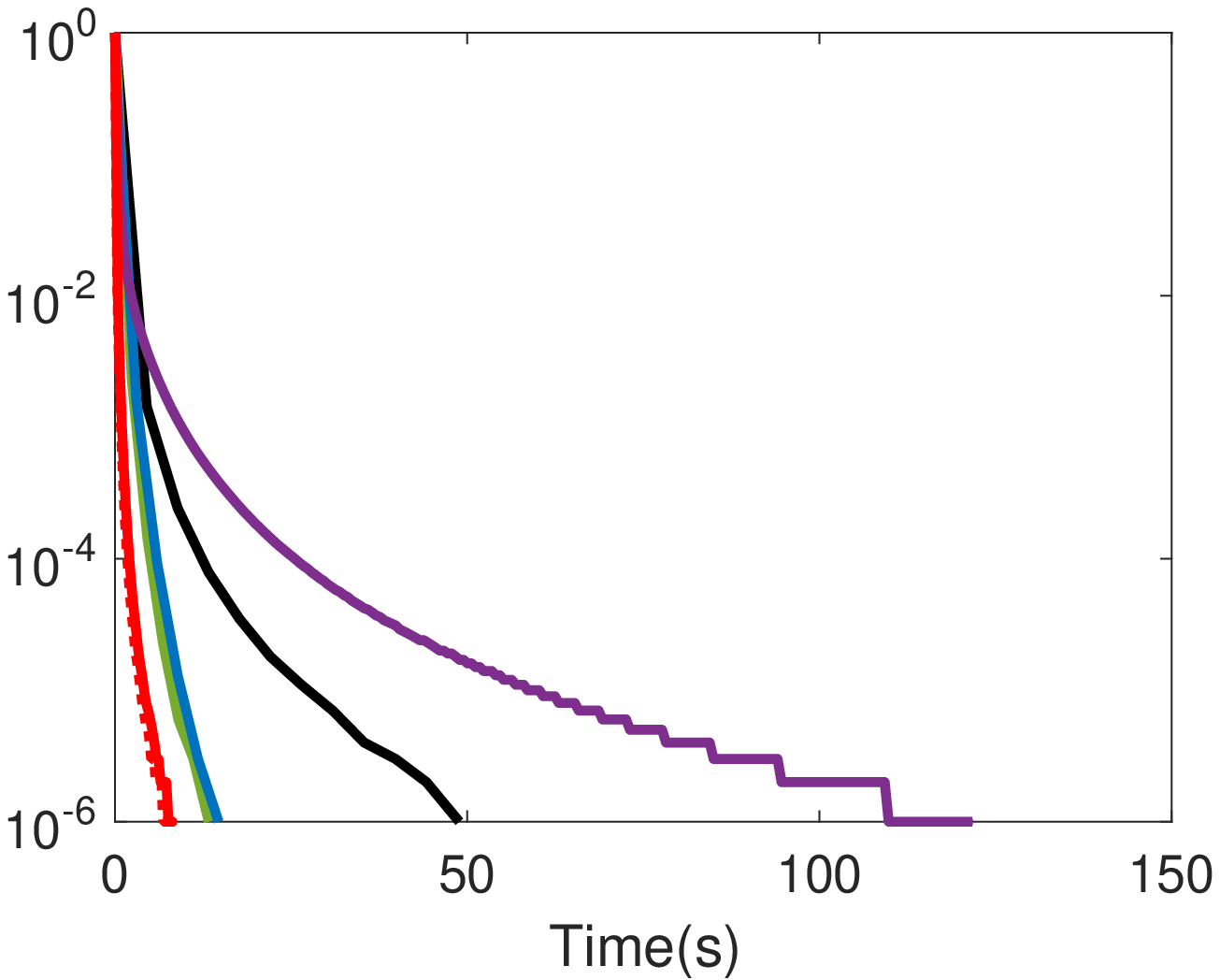} &
\includegraphics[width=0.25\linewidth]{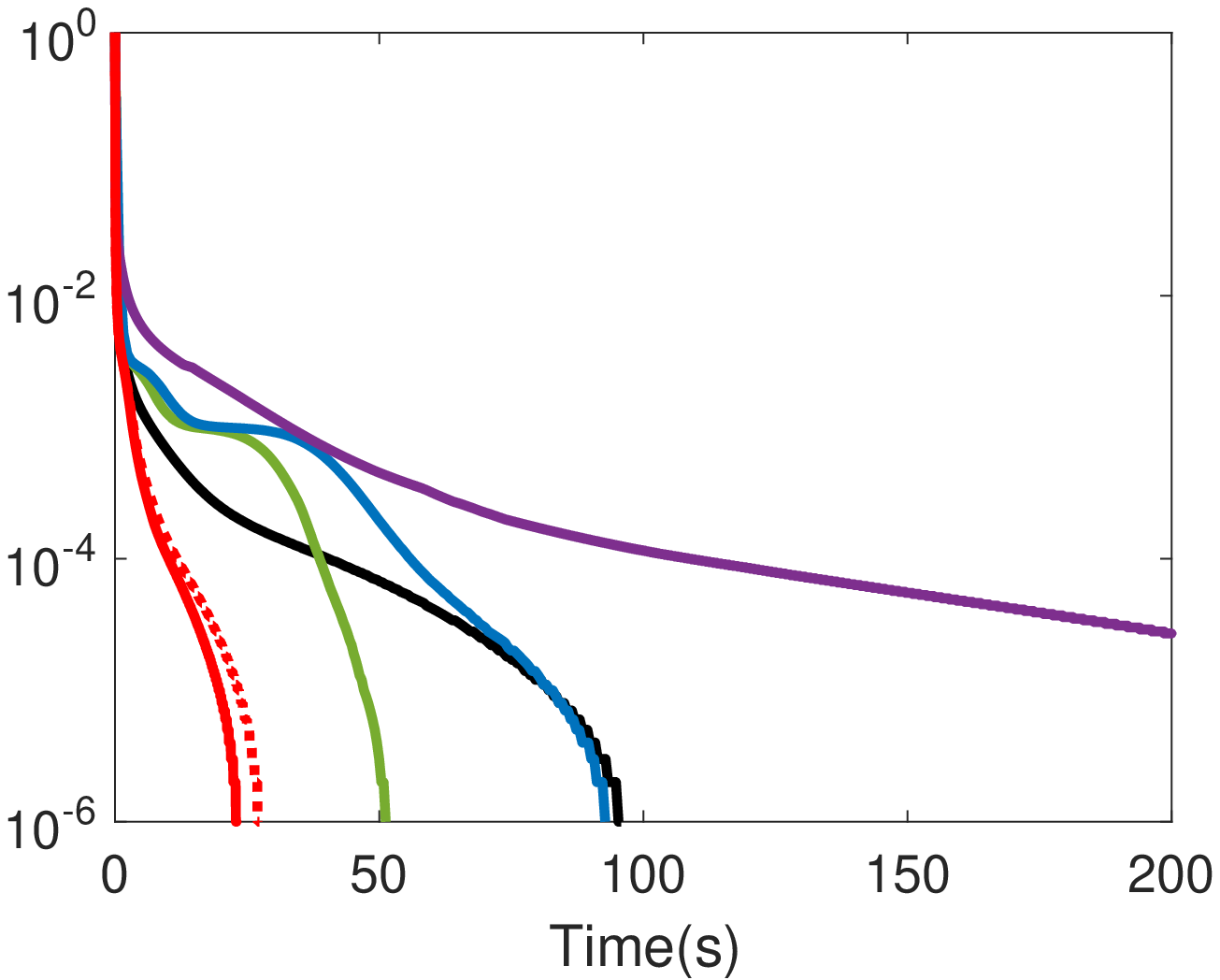} & 
\includegraphics[width=0.25\linewidth]{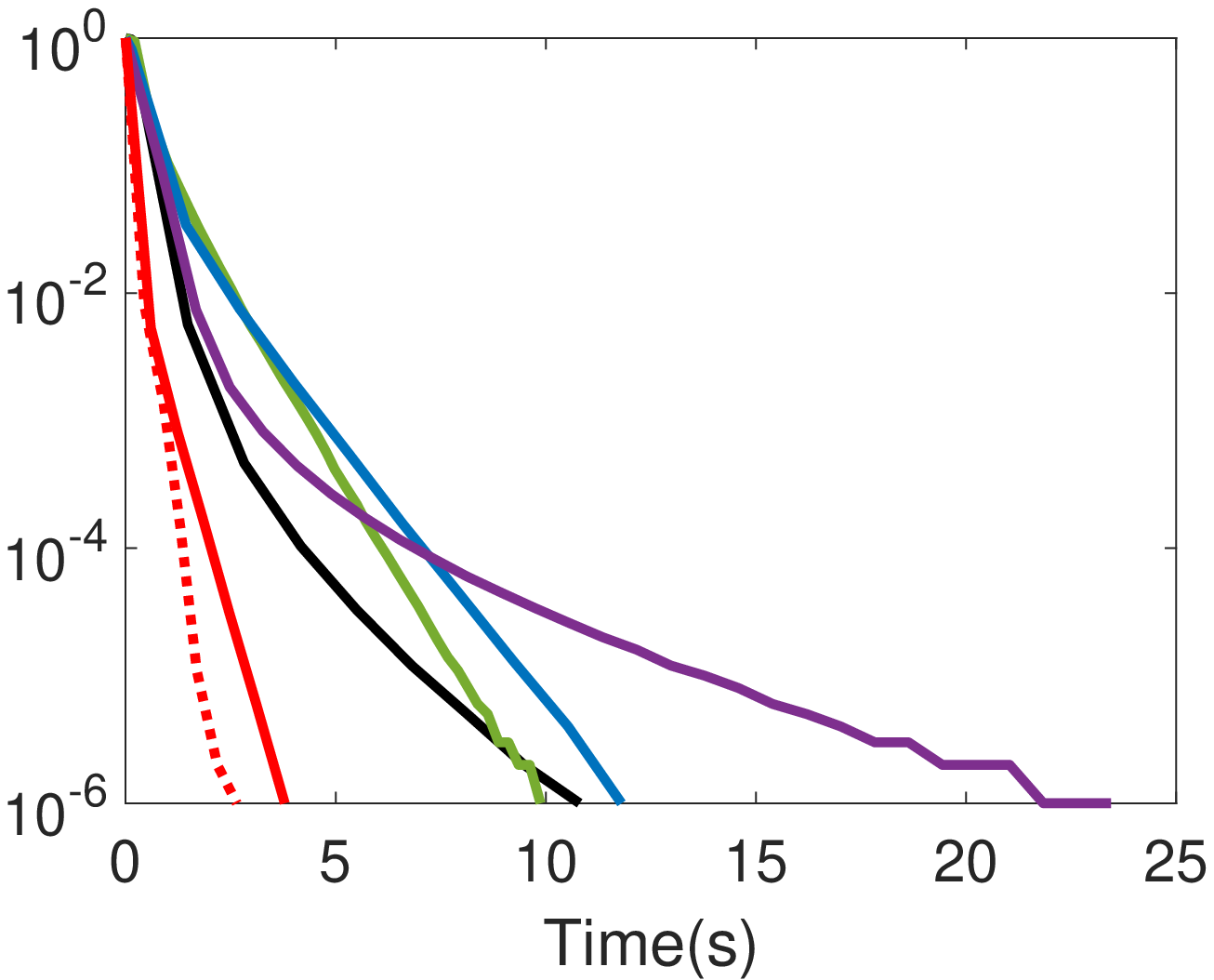} 
\end{tabular}
\caption{Comparison of the runtime efficiency ($\rho_1 - \w_t^\top \A \w_t$ vs. run time in seconds) of various algorithms for computing the leading eigenvector on real-world network datasets.}
\label{fig:realdata_experiments}
\end{figure*}

The results are summarized in Figure~\ref{fig:realdata_experiments}. We observe that the proposed methods usually improves over CPM and SGCD, with up to $5$x speedup in runtime over CPM/SGD depending on the dataset. Different from the simulation study, though SI-ACDM have faster convergence in theory, it is often slightly slower than SI-GSL due to more maintainance for each update (it needs to maintain two vector sequences instead of one for plain CD methods).

%% file: append.tex
\appendix

\section{Auxiliary Lemmas}
\label{sec:auxiliary}

\begin{lem} \label{lem:alignment-to-objective}
Consider a positive semidefinite matrix $\A \in \bbR^{d\times d}$ with eigenvalues $\rho_1 \ge  \dots \ge \rho_d \ge 0$ and corresponding eigenvectors $\p_1,\dots,\p_d$. For a unit vector $\vv$ satisfying $\vv^\top \p_1 \ge 1 - \epsilon$ where $\epsilon\in (0,1)$, we have
\begin{align*}
\vv^\top \A \vv \ge \rho_1 (1-2\epsilon).
\end{align*}
\end{lem}
\begin{proof} By direct calculation, we obtain
\begin{align*}
\vv^\top \A \vv & = \vv^\top \left( \sum_{i=1}^d  \rho_i \p_i \p_i^\top \right) \vv \\
& =  \sum_{i=1}^d  \rho_i \left(\vv^\top \p_i\right)^2 \\
& \ge  \rho_1 (\vv^\top \p_1)^2 \\
& \ge \rho_1 (1-\epsilon)^2 \\
& \ge \rho_1 (1-2\epsilon).
\end{align*}
\end{proof}

\begin{lem} \label{lem:normalized-difference}
For two nonzero vectors $\aa, \b \in \bbR^{d}$, we have
\begin{align*}
\norm{ \frac{\aa}{\norm{\aa}} - \frac{\b}{\norm{\b}}} \le 2 \frac{\norm{\aa-\b}}{\norm{\aa}}.
\end{align*}
\end{lem}
\begin{proof} By direct calculation, we obtain
\begin{align*}
\norm{ \frac{\aa}{\norm{\aa}} - \frac{\b}{\norm{\b}}} & \le 
\norm{ \frac{\aa}{\norm{\aa}} - \frac{\b}{\norm{\aa}}} +
\norm{ \frac{\b}{\norm{\aa}} - \frac{\b}{\norm{\b}}} \\
& =  \frac{ \norm{\aa - \b}}{\norm{\aa}} +
\norm{\b} \cdot  \frac{ \abs{ \norm{\aa} - \norm{\b}}}{\norm{\aa} \norm{\b}}\\
& \le \frac{ \norm{\aa - \b}}{\norm{\aa}} + \frac{ \norm{\aa - \b}}{\norm{\aa}} \\
& = 2 \frac{\norm{\aa-\b}}{\norm{\aa}}
\end{align*}
where we have used the triangle inequality in the second inequality.
\end{proof}

\section{Analysis of inexact power method in different regimes}
\label{sec:iteration-complexity-inexact-power-method-appendix}

\begin{lem}[Iteration complexity of inexact power method]
  \label{lem:iteration-complexity-inexact-power-method}
  Consider the following inexact shift-and-invert power iterations: for $t=1,\dots,$
  \begin{gather*}
    \tw_{t} \approx  \argmin_{\w}\; f_t (\w) = \frac{1}{2} \w^\top (\lambda \I - \A) \w - \w_{t-1}^\top \w,  \\
    \w_{t} \leftarrow \frac{\tw_{t}}{\norm{\tw_{t}}},
  \end{gather*}
  where $f_t (\tw_{t}) \le \min_{\w} f_t (\w) + \epsilon_t$. 
  \begin{itemize}
  \item (Crude regime) Define $T_1=\ceil{ \frac{2}{\epsilon} \log \left( \frac{4}{\epsilon \xi_{01}^2} \right) }$. Assume that for all $t=1,\dots,T_1$, the error in minimizing  $f_t (\w)$ satisfies
    \begin{align*} 
      \epsilon_t \le \frac{\epsilon^2 \beta_d^2}{128 \beta_1} \left( \frac{(2 \beta_1 / \beta_d)-1}{(2 \beta_1 / \beta_d)^{T_1} -1} \right)^2.
    \end{align*}
    Then we have $\w_{T_1}^\top (\lambda - \A)^{-1} \w_{T_1} \ge (1-\epsilon) \beta_1$.

  \item (Accurate regime) Define $G_0 := \frac{ \sqrt{ \sum_{i=2}^d  \xi_{0i}^2 / \beta_i }}{ \sqrt{ \xi_{01}^2 / \beta_1}}$. Assume that for all $t\ge 1$, the error in minimizing  $f_t (\w)$ satisfies
    \begin{align*} 
      \epsilon_t \le  \min \left( {\sum_{i=2}^d \xi_{(t-1)i}^2 / \beta_i},\ {\xi_{(t-1)1}^2 / \beta_1} \right) \cdot  \frac{\left( \beta_1 - \beta_2 \right)^2}{32}.
    \end{align*}
    Let $\gamma = \frac{3 \beta_1 + \beta_2}{\beta_1 + 3 \beta_2} > 1$ and $T_2=\ceil{\frac{1}{2} \log_{\gamma} \left( \frac{G_0^2}{\epsilon} \right)}$. Then we have $\w_t^\top \p_1 \ge 1 - \epsilon$ for all $t\ge T_2$.
  \end{itemize}
\end{lem}
\begin{proof}
  In the following, we use the shorthand $\M_\lambda = (\lambda \I - \A)^{-1}$. Observe that
  \begin{align*}
    \p_i^\top {\M}_{\lambda} \p_i &= \beta_i, \qquad\text{for} \quad i=1,\dots,d, \\
    \p_i^\top {\M}_{\lambda} \p_j &= 0, \qquad\;\,\text{for} \quad i\neq j. 
  \end{align*}

  Let the exact solution to the least squares problem $\min_{\w}\; f_t (\w)$ be $\tw_{t}^* = \M_{\lambda} \w_{t-1}$. If we obtain an approximate solution $\tw_{t}$ such that  $f_{t} (\tw_{t}) - f_{t} (\tw_{t}^*) = \epsilon_t$, it follows from the quadratic objective that 
  \begin{align} \label{e:error-normM}
    \epsilon_t  & = \frac{1}{2} \left(\tw_{t} - \tw_{t}^* \right)^\top  {\M}_{\lambda}^{-1} \left(\w_{t} - \w_{t}^* \right) = \frac{1}{2} \normM{\tw_{t} - \tw_{t}^*}^2.
  \end{align}
  Furthermore, we have for the exact solution that 
  \begin{align} \label{e:lsq-exact-solution}
    \tw_{t}^* = {\M}_{\lambda} \w_{t-1} = {\M}_{\lambda} \sum_{i=1}^d \xi_{(t-1)i} \p_i = 
    \sum_{i=1}^d \beta_i \xi_{(t-1)i} \p_i.
  \end{align}

  \paragraph{Crude regime} 
  For the crude regime, we denote by $n$ the number of eigenvalues of $\M_{\lambda}$ that are greater than or equal to $\left(1 - \frac{\epsilon}{4} \right) \beta_1$.

  We will relate the iterates of inexact power method to those of the exact power method
  \begin{align*}
    \tv_t \leftarrow \M_{\lambda} \vv_{t-1}, \qquad \vv_t \leftarrow \frac{\tv_t}{\norm{\tv_t}}, \qquad t=1,\dots
  \end{align*}
  from the same initialization $\vv_0 = \w_0$.

  On the one hand, we can show that exact power iterations converges quickly for eigenvalue estimation. Since $\vv_t = \frac{\M_{\lambda}^t \vv_0}{\norm{\M_{\lambda}^t \vv_0}}$ and $\M_{\lambda}^t \vv_0 = \sum_{i=1}^d \beta_i^{t} \xi_{0i} \p_i$, we have
  \begin{align*}
    \sum_{i=n+1}^d (\vv_t^\top \p_i)^2 & 
    = \frac{{\sum_{i=n+1}^d \beta_i^{2t} \xi_{0i}^2}}{{\sum_{i=1}^d \beta_i^{2t} \xi_{0i}^2}} 
    \le \frac{ \beta_{n+1}^{2t} \sum_{i=n+1}^d \xi_{0i}^2  }{\beta_1^{2t} \xi_{01}^2} \\
    & \le \frac{1}{\xi_{01}^2} \left( 1 - \frac{\epsilon}{4} \right)^{2t} 
    \le \frac{e^{-\epsilon t / 2}}{\xi_{01}^2}.
  \end{align*}
  This indicates that after $T_1 = \ceil{ \frac{2}{\epsilon} \log \left( \frac{4}{\epsilon \xi_{01}^2} \right)  } $ iterations, we have $\sum_{i=n+1}^d (\vv_t^\top \p_i)^2 \le \frac{\epsilon}{4}$ for all $t\ge T_1$. And as a result, for all $t\ge T_1$, it holds that
  \begin{align*}
    & \vv_t^\top \M_{\lambda} \vv_t \\
    =\ & \vv_t^\top \left(\sum_{i=1}^d \beta_i \p_i \p_i^\top \right) \vv_t = \sum_{i=1}^d \beta_i (\vv_t^\top \p_i)^2 \\
    \ge\ & \beta_{n} \sum_{i=1}^n (\vv_t^\top \p_i)^2 
    \ge \left( 1 - \frac{\epsilon}{4} \right) \beta_1 \cdot \left(1 - \sum_{i=n+1}^d (\vv_t^\top \p_i)^2 \right) \\
    \ge\ &  \left( 1 - \frac{\epsilon}{4} \right)^2 \beta_1 
    \ge \left( 1 - \frac{\epsilon}{2} \right) \beta_1.
  \end{align*}

  On the other hand, we can lower bound $\w_t^\top \M_{\lambda} \w_t$ using $\vv_t^\top \M_{\lambda} \vv_t$:
  \begin{align*}
    \w_t^\top \M_{\lambda} \w_t & = (\vv_t + \w_t - \vv_t)^\top \M_{\lambda} (\vv_t + \w_t - \vv_t) \\
    & \ge \vv_t^\top \M_{\lambda} \vv_t + 2 (\w_t - \vv_t)^\top  \M_{\lambda} \vv_t \\
    & \ge \vv_t^\top \M_{\lambda} \vv_t - 2 \beta_1 \norm{\w_t - \vv_t} \\
    & \ge \left( 1 - \frac{\epsilon}{2} \right) \beta_1 - 2 \beta_1 \norm{\w_t - \vv_t}.
  \end{align*}
  Therefore, we will have $\w_t^\top \M_{\lambda} \w_t \ge \left( 1 - \epsilon \right) \beta_1$ as desired if we can guarantee $\norm{\w_{T_1} - \vv_{T_1}} \le \frac{\epsilon}{4}$.

  We now upper bound the difference $S_t:=\norm{\tw_t - \tv_t}$ by induction.
  Due to the assumption of same initialization, we have $S_0=0$. For $t\ge 2$, we can decompose the error into two terms:
  \begin{align*}
    \norm{\tw_t - \tv_t} & \le \norm{\tw_t - \M_{\lambda} \w_{t-1}} + \norm{\M_{\lambda} \w_{t-1} - \tv_t} \\
    & \le \norm{\tw_t - \M_{\lambda} \w_{t-1}} + \norm{\M_{\lambda} \w_{t-1} - \M_{\lambda} \vv_{t-1}}.
  \end{align*}
  The first term concerns the error from inexact minimization of $f_t (\w)$ and can be bounded using~\eqref{e:error-normM}:
  \begin{align}\label{e:crude-errors-1}
    \norm{\tw_t - \M_{\lambda} \w_{t-1}} \le \norm{\M_{\lambda}^{\frac{1}{2}}} \cdot \normM{\tw_{t} - \tw_{t}^*} \le  \sqrt{2 \beta_1 \epsilon_t}.
  \end{align}
  The second term concerns the error from inexact target:
  \begin{gather}
    \norm{\M_{\lambda} \w_{t-1} - \M_{\lambda} \vv_{t-1}} \le \norm{\M_{\lambda}} \cdot \norm{\w_{t-1} - \vv_{t-1}} \nonumber \\  \label{e:crude-errors-2}
    = \beta_1 \norm{ \frac{\tw_{t-1}}{\norm{\tw_{t-1}}} - \frac{\tv_{t-1}}{\norm{\tv_{t-1}}} } \le 2 \beta_1 \frac{\norm{\tw_{t-1} - \tv_{t-1}}}{\norm{\tv_{t-1}}}
  \end{gather}
  where the last inequality is due to Lemma~\ref{lem:normalized-difference}. 

  Noting that $\norm{\tv_{t-1}} = \norm{\M_{\lambda} \vv_{t-2}} \ge \beta_d \norm{\vv_{t-2}} = \beta_d$, and combining~\eqref{e:crude-errors-1} and~\eqref{e:crude-errors-2}, we obtain
  \begin{align*}
    S_t \le \sqrt{2 \beta_1 \epsilon_t} + \frac{2 \beta_1}{\beta_d} S_{t-1}.
  \end{align*}
  Fixing $\epsilon_t=\epsilon^\prime$ for all $t$ and unfolding the above inequality over $t$ yield
  \begin{align*}
    S_{T_1} \le \sqrt{2 \beta_1 \epsilon^\prime} \cdot \sum_{t=0}^{T_1} \left( \frac{2 \beta_1}{\beta_d}\right)^t = \sqrt{2 \beta_1 \epsilon^\prime} \cdot \frac{(2 \beta_1 / \beta_d)^{T_1} -1}{(2 \beta_1 / \beta_d)-1}.
  \end{align*}
  By Lemma~\ref{lem:normalized-difference} again, we have
  \begin{align*}
    \norm{\w_{t} - \vv_{t}} & = \norm{ \frac{\tw_{t}}{\norm{\tw_{t}}} - \frac{\tv_{t}}{\norm{\tv_{t}}} } \le 2 \frac{\norm{\tw_{t} - \tv_{t}}}{\norm{\tv_{t}}} \le \frac{2  S_t}{\beta_d} \\
 & \le \frac{\sqrt{8 \beta_1 \epsilon^\prime}}{\beta_d} \cdot \frac{(2 \beta_1 / \beta_d)^{T_1} -1}{(2 \beta_1 / \beta_d)-1}. 
  \end{align*}
  Setting the RHS to be $\frac{\epsilon}{4}$ gives 
  \begin{align*}
    \epsilon^\prime =  \frac{\epsilon^2 \beta_d^2}{128 \beta_1} \left( \frac{(2 \beta_1 / \beta_d)-1}{(2 \beta_1 / \beta_d)^{T_1} -1} \right)^2.
  \end{align*}
  \paragraph{Accurate regime} In the accurate regime, the potential function we use to evaluate the progress of each iteration is
  \begin{align*}
    G (\w_{t})  = \frac{\normM{\calP_{\perp} \w_{t}}}{\normM{\calP_{\parallel} \w_{t}}} 
    = \frac{ \sqrt{ \sum_{i=2}^d  \xi_{ti}^2 / \beta_i }}{ \sqrt{ \xi_{t1}^2 / \beta_1}},
  \end{align*}
  where $\calP_{\perp}$ and $\calP_{\parallel}$ denote projections onto the subspaces that are orthogonal and parallel to $\p_1$ respectively. 
  Note that 
  \begin{align*}
    \abs{ \sin \theta_t } = \sqrt{ \sum_{i=2}^d  \xi_{ti}^2 } \le 
    \abs{ \tan \theta_t } = \frac{ \sqrt{ \sum_{i=2}^d  \xi_{ti}^2 }}{ \sqrt{ \xi_{t1}^2}}  
    \le G (\w_{t})
  \end{align*}
  where $\theta_t$ is the angle between $\w_t$ and $\p_1$.

  Our goal in this regime is to have $\abs{\sin \theta_t} \le \sqrt{\epsilon}$, as this implies 
  \begin{align*}
    \w_t^\top \p_1 = \cos \theta_t = \sqrt{ 1 - \sin^2 \theta_t  }
    \ge 1 - \sin^2 \theta_t \ge 1 - \epsilon 
  \end{align*}
  as desired. We ensure $\abs{\sin \theta_t} \le \sqrt{\epsilon}$ by requiring that $G (\w_{t}) \le \sqrt{\epsilon}$. 

  In view of~\eqref{e:error-normM} and~\eqref{e:lsq-exact-solution}, we can bound the numerator and denominator of $G(\w_{t})$ respectively with regard to $G(\w_{t-1})$:
  \begin{align*}
    & \normM{\calP_{\perp} \frac{\tw_{t}}{\norm{\tw_{t}}} } \\
    \le\ & \frac{1}{\norm{\tw_{t}}}  \left( \normM{\calP_{\perp} \tw_{t}^* } + \normM{\calP_{\perp} \left( \tw_{t} - \tw_{t}^* \right) } \right) \\
    \le\ & \frac{1}{\norm{\tw_{t}}} \left( \normM{\sum_{i=2}^d \beta_i \xi_{(t-1)i} \p_i } + \normM{ \tw_{t} - \tw_{t}^* } \right) \\
    =\ & \frac{1}{\norm{\tw_{t}}} \left( \sqrt{\sum_{i=2}^d \beta_i \xi_{(t-1)i}^2}  + \sqrt{2 \epsilon_t } \right),
  \end{align*}
  and 
  \begin{align*}
    & \normM{\calP_{\parallel} \frac{\tw_{t}}{\norm{\tw_{t}}} } \\
    \ge\ & \frac{1}{\norm{\tw_{t}}} \left( \normM{\calP_{\parallel} \tw_{t}^* } - \normM{\calP_{\parallel} \left( \tw_{t} - \tw_{t}^* \right) } \right) \\
    \ge\ & \frac{1}{\norm{\tw_{t}}} \left(  \normM{\beta_1 \xi_{(t-1)1} \p_1} - \normM{ \tw_{t} - \tw_{t}^* } \right) \\
    =\ & \frac{1}{\norm{\tw_{t}}} \left( \sqrt{\beta_1 \xi_{(t-1)1}^2} - \sqrt{2 \epsilon_t} \right).
  \end{align*}
  
  Consequently, we have
  \begin{align*}
    G(\w_{t}) 
    & \le \frac{\sqrt{\sum_{i=2}^d \beta_i \xi_{(t-1)i}^2} + \sqrt{2 \epsilon_t }}{ \sqrt{\beta_1 \xi_{(t-1)1}^2} - \sqrt{2 \epsilon_t }} \\
    & \le \frac{\beta_2 \sqrt{\sum_{i=2}^d \xi_{(t-1)i}^2 / \beta_i} + \sqrt{2 \epsilon_t} }{\beta_1 \sqrt{\xi_{(t-1)1}^2 / \beta_1} - \sqrt{2 \epsilon_t}} \\
    & = G(\w_{t}) \cdot \frac{\beta_2  + \frac{\sqrt{2 \epsilon_t}}{\sqrt{\sum_{i=2}^d \xi_{(t-1)i}^2 / \beta_i}}}{\beta_1  - \frac{\sqrt{2 \epsilon_t}}{\sqrt{\xi_{(t-1)1}^2 / \beta_1}}}.
  \end{align*}

  As long as 
\begin{align*}
\sqrt{2 \epsilon_t} \le \min \left( \sqrt{\sum_{i=2}^d \xi_{(t-1)i}^2 / \beta_i},\ \sqrt{\xi_{(t-1)1}^2 / \beta_1}  \right) \cdot \frac{\beta_1 - \beta_2}{4},
\end{align*}
or equivalently
  \begin{align*}
    \epsilon_t \le  \min \left( {\sum_{i=2}^d \xi_{(t-1)i}^2 / \beta_i},\ {\xi_{(t-1)1}^2 / \beta_1} \right) \cdot  \frac{\left( \beta_1 - \beta_2 \right)^2}{32} ,
  \end{align*}
  we are guaranteed that
  \begin{align*}
    G(\w_{t}) \le \frac{\beta_1 + 3 \beta_2}{3 \beta_1 + \beta_2} \cdot G(\w_{t-1}).
  \end{align*}
  And when this holds for all $t\ge 1$, the sequence $\{G(\w_{t})\}_{t=0,\dots}$ decreases (at least) at a constant geometric rate of $\frac{\beta_1 + 3 \beta_2}{3 \beta_1 + \beta_2} < 1 $. Therefore, the number of iterations needed to achieve $\abs{\sin \theta_{T_2}} \le \sqrt{\epsilon}$ is $T_2 = \ceil{ \log_{\gamma} \left( \frac{G(\w_0)}{\sqrt{\epsilon}} \right) } = \ceil{ \frac{1}{2} \log_{\gamma} \left( \frac{G_0^2}{\epsilon} \right) }$ for $\gamma=\frac{3 \beta_1 + \beta_2}{\beta_1 + 3 \beta_2}$.

\end{proof}

\section{Bounding initial error for least squares}
\label{sec:shift-and-invert-warm-start-append}

For each least squares problem $f_t (\w)$ in inexact shift-and-invert power method, we can use $\tw_{t-1}$ from the previous iteration as initialization, and the intuition behind is that as the power power method converges, the least squares problems become increasingly similar. However, an even better warm-start is to use as initialization $\alpha_t \w_{t-1}$ for some $\alpha_t$ and minimize the initial suboptimality over $\alpha_t$~\citep{Garber_16a,Ge_16a}. Let the exact solution to the least squares problem $\min_{\w}\; f_t (\w)$ be $\tw_{t}^* = (\lambda \I - \A)^{-1} \w_{t-1}$, which gives the optimal objective function value $f_t^* = - \frac{1}{2} \w_{t-1}^\top (\lambda \I - \A)^{-1} \w_{t-1}$. 
Observe that
\begin{align*}
  \epsilon_t^{init} & = f_{t} (\alpha_t \w_{t-1}) - f_t^* \\
  & = \frac{\alpha_t^2}{2} \cdot \w_{t-1}^\top (\lambda \I - \A) \w_{t-1} - \alpha_t \cdot \w_{t-1}^\top \w_{t-1} - f_t^*.
\end{align*}
This is a quadratic function of $\alpha_t$, with minimum achieved at $\alpha_t=\frac{1}{\w_{t-1}^\top (\lambda \I - \A) \w_{t-1}}$. 
With this choice of $\alpha_t$, we obtain
\begin{align}
  & \qquad \epsilon_t^{init} \le f_t (\beta_1 \w_{t-1}) - f_t^*  \nonumber \\
  & =  \frac{\beta_1^2}{2} \cdot \w_{t-1}^\top (\lambda \I - \A) \w_{t-1} - \beta_1 + \frac{1}{2} \w_{t-1}^\top (\lambda \I - \A)^{-1} \w_{t-1} \nonumber \\
  & = \frac{\beta_1^2}{2} \sum_{i=1}^d \xi_{(t-1)i}^2 / \beta_i  - \beta_1 \sum_{i=1}^d \xi_{(t-1)i}^2 + \frac{1}{2} \sum_{i=1}^d \beta_i \xi_{(t-1)i}^2 \nonumber \\
  & \le \frac{1}{2} \sum_{i=1}^d \left( \beta_1 - \beta_i \right)^2 \cdot \xi_{(t-1)i}^2 / \beta_i \nonumber \\
  & \le \frac{\beta_1^2}{2} \sum_{i=2}^d \xi_{(t-1)i}^2 / \beta_i. \label{e:warm-start-accurate}
\end{align}

As we show in the next lemma, this initialization will allow the ratio between initial and final error to be a constant in the accurate regime, independent of the final suboptimality in alignment. 

In the crude regime, we simply use the rough bound 
\begin{align}
  \epsilon_t^{init} & \le f_t (\0) - f_t^* = \frac{1}{2} \w_{t-1}^\top (\lambda \I - \A)^{-1} \w_{t-1} \nonumber \\
& = \frac{1}{2} \sum_{i=1}^d \beta_i \xi_{(t-1)i}^2  \le  \frac{\beta_1}{2}. \label{e:warm-start-crude}
\end{align}

Substituting~\eqref{e:warm-start-accurate} and~\eqref{e:warm-start-crude} into Lemma~\ref{lem:iteration-complexity-inexact-power-method} yields Lemma~\ref{lem:ft_ratio}.

\section{Proof of Lemma~\ref{lem:shift-and-invert-repeat-until}}

\begin{proof}

  Observe that the eigenvalues of $\M_{\lambda}=(\lambda \I - \A)^{-1}$ are functions of $\lambda$ as $\lambda$ changes over iterations. Below we use $\sigma_i \left( \M_{\lambda} \right)$ to denote the i-th eigenvalue of $\M_{\lambda}$ instead of $\beta_1,\dots,\beta_d$ to make this dependence explicit.

  Define the condition number of $\M_{\lambda}$ as
  \begin{align*}
    \kappa_{\lambda} : = \frac{\sigma_1(\M_{\lambda})}{\sigma_d(\M_{\lambda})} = \frac{\frac{1}{\lambda-\rho_1}}{\frac{1}{\lambda-\rho_d}} = \frac{\lambda - \rho_d}{\lambda - \rho_1},
  \end{align*}

  Suppose we have the upper and lower bound of all  $\sigma_1 (\M_{\lambda_{(s)}})$ used in the \textbf{repeat-until} loop: 
  \begin{align*}
    \overline{\sigma} \ge \sigma_1 (\M_{\lambda_{(s)}}), \quad \underline{\sigma} \ge \sigma_1 (\M_{\lambda_{(s)}}), \quad \text{for} \;  s=1,\dots,
  \end{align*}
  whose specific values will be provided shortly.

  Throughout the loop, we require for all iteration $s$ that 
  \begin{align} \label{e:delta-s-1}
    \sqrt{ 2 \overline{\sigma} \te_s } \le \frac{\underline{\sigma}}{8}. 
  \end{align}

  Let the power iterations for different $s$ start from the same initialization $\w_0$ with $\xi_{01} = \w_0^\top \p_1$, and apply Lemma~\ref{lem:iteration-complexity-inexact-power-method} (crude regime) with $\epsilon=\frac{1}{4}$. By our choice of $m_1$ and setting the ratio 
  \begin{align}  \label{e:epsilon-1}
    \frac{\epsilon_t^{init}}{\epsilon_t} =  1024 \kappa_{\lambda_{(s)}}^2 \left( \frac{ (2 \kappa_{\lambda_{(s)}})^{m_1} -1}{2 \kappa_{\lambda_{(s)}} - 1 } \right)^2
  \end{align}
  according to Lemma~\ref{lem:ft_ratio} (crude regime), we obtain 
  \begin{align} \label{e:delta-s-2}
    \w_{s m_1}^\top \M_{\lambda_{(s-1)}} \w_{s m_1} \ge \frac{3}{4} \sigma_1 (\M_{\lambda_{(s-1)}}).
  \end{align}

  In view of the definition of the vector $\uu_s$, and following the same argument in~\eqref{e:crude-errors-1}, we have
  \begin{align*}
    \norm{ {\uu_s} - \M_{\lambda_{(s-1)}} \w_{s m_1} } \le \sqrt{ 2 \sigma_1 (\M_{\lambda_{(s-1)}}) \cdot \te_s}. 
  \end{align*}

  Then for every iteration $s$ of \textbf{repeat-until}, it holds that
  \begin{align*}
    & \w_{s m_1}^\top \uu_s  \\
    =\ & \w_{s m_1}^\top \M_{\lambda_{(s-1)}} \w_{s m_1} + \w_{s m_1}^\top \left( \uu_s - \M_{\lambda_{(s-1)}} \w_{s m_1} \right)  \\
    \in\, & \left[ \w_{s m_1}^\top \M_{\lambda_{(s-1)}} \w_{s m_1} - \sqrt{ 2 \sigma_1 (\M_{\lambda_{(s-1)}}) \cdot \te_s },\ \right.\\
    & \qquad \left. \w_{s m_1}^\top \M_{\lambda_{(s-1)}} \w_{s m_1} + \sqrt{ 2 \sigma_1 (\M_{\lambda_{(s-1)}}) \cdot \te_s } \right] \\
    \in\, & \left[ \w_{s m_1}^\top \M_{\lambda_{(s-1)}} \w_{s m_1} - \sqrt{ 2 \overline{\sigma} \te_s },\ \right.\\
   & \qquad \left. \w_{s m_1}^\top \M_{\lambda_{(s-1)}} \w_{s m_1} + \sqrt{ 2 \overline{\sigma} \te_s } \right] \\
    \in\, & \left[ \w_{s m_1}^\top \M_{\lambda_{(s-1)}} \w_{s m_1} - \frac{\underline{\sigma}}{8},\ \w_{s m_1}^\top \M_{\lambda_{(s-1)}} \w_{s m_1} + \frac{\underline{\sigma}}{8} \right],
  \end{align*}
  where we have used the Cauchy-Schwarz inequality in the second step and~\eqref{e:delta-s-1} in the last step.

  In view of~\eqref{e:delta-s-1} and~\eqref{e:delta-s-2}, it follows that
  \begin{align*}
    & \w_{s m_1}^\top \uu_s - \underline{\sigma}/8  \\
    \in\ & \left[ \w_{s m_1}^\top \M_{\lambda_{(s-1)}} \w_{s m_1} - \frac{\underline{\sigma}}{4},\ \w_{s m_1}^\top \M_{\lambda_{(s-1)}} \w_{s m_1} \right] \\
    \in\ & \left[ \frac{3}{4} \sigma_1 (\M_{\lambda_{(s-1)}}) - \frac{\underline{\sigma}}{4},\ \w_{s m_1}^\top \M_{\lambda_{(s-1)}} \w_{s m_1} \right] \\
    \in\ & \left[ \frac{1}{2} \sigma_1 (\M_{\lambda_{(s-1)}}),\ \sigma_1 (\M_{\lambda_{(s-1)}})\right].
  \end{align*}

  By the definition of $\Delta_s$ in Algorithm~\ref{alg:meta-shift-and-invert} and the fact that $\sigma_1 (\M_{\lambda_{(s-1)}}) = \frac{1}{\lambda_{(s-1)}-\rho_1}$, we have 
  \begin{align} 
    \Delta_s  & = \frac{1}{2}\cdot \frac{1}{\w_{s m_1}^\top \uu_s - \underline{\sigma}/8 } \nonumber \\
    & \in \left[ \frac{1}{2} \left( \lambda_{(s-1)}-\rho_1 \right),\ \lambda_{(s-1)}-\rho_1 \right]. \label{e:lambda-recurse}
  \end{align}

  And as a result,
  \begin{align*}
    \lambda_{(s)} & = \lambda_{(s-1)} - \frac{\Delta_s}{2} 
    \ge \lambda_{(s-1)} - \frac{1}{2} \left( \lambda_{(s-1)} - \rho_1 \right) \\
    & = \frac{\lambda_{(s-1)} + \rho_1}{2},
  \end{align*}
  and thus by induction (note $\lambda_{(0)} \ge \rho_1$) we have $\lambda_{(s)} \ge \rho_1$ throughout the \textbf{repeat-until} loop.

  From~\eqref{e:lambda-recurse} we also obtain
  \begin{align*}
    \lambda_{(s)} - \rho_1 & =  \lambda_{(s-1)}  - \rho_1  - \frac{\Delta_s}{2} \\
    & \le \lambda_{(s-1)}  - \rho_1 - \frac{1}{4} \left( \lambda_{(s-1)} - \rho_1 \right) \\
    & = \frac{3}{4} \left( \lambda_{(s-1)} - \rho_1 \right).
  \end{align*}

  To sum up, $\lambda_{(s)}$ approaches $\rho_1$ from above and the gap between $\lambda_{(s)}$ and $\rho_1$ reduces at the geometric rate of $\frac{3}{4}$. Thus after at most $T_3=\ceil{ \log_{3/4} \left( \frac{\tilde{\Delta}}{\lambda_{(0)} - \rho_1} \right) } = \calO\left( \log \left( \frac{1}{\tilde{\Delta}} \right) \right)$ iterations, we reach a $\lambda_{(T_3)}$ such that $\lambda_{(T_3)} - \rho_1 \le \tilde{\Delta}$. And in view of~\eqref{e:lambda-recurse}, the \textbf{repeat-until} loop exits in the next iteration. Hence, the overall number of iterations is at most $T_3 + 1 = \calO \left( \frac{1}{\tilde{\Delta}} \right)$.

  We now analyze $\lambda_{(f)}$ and derive the interval it lies in. 
  Note that $\Delta_{f} \le \tilde{\Delta}$ and $\Delta_{f-1} > \tilde{\Delta}$ by the exiting condition. In view of~\eqref{e:lambda-recurse}, we have
  \begin{align*}
    \lambda_{(f)} - \rho_1 & = \lambda_{(f-1)} - \rho_1 - \frac{\Delta_{f}}{2} \le 2 \Delta_f  - \frac{\Delta_{f}}{2} \\
    & = \frac{3 \Delta_f}{2} \le \frac{3 \tilde{\Delta}}{2}.
  \end{align*}

  On the other hand, we have
  \begin{align} 
    \lambda_{(f)} - \rho_1 & = \lambda_{(f-1)} - \rho_1 - \frac{\Delta_{f}}{2} \nonumber \\
& \ge \lambda_{(f-1)} - \rho_1 -  \frac{1}{2} \left( \lambda_{(f-1)} - \rho_1 \right) \nonumber \\
& = \frac{1}{2} \left( \lambda_{(f-1)} - \rho_1 \right). \label{e:lambda-recurse-2}
  \end{align}
  If $f=1$, then by our choice of $\lambda_{(0)}$ we have that $\lambda_{(f)} - \rho_1 \ge \tilde{\Delta}$. Otherwise, by unfolding~\eqref{e:lambda-recurse-2} one more time, we have that
  \begin{align*} 
    \lambda_{(f)} - \rho_1 \ge \frac{1}{4} \left( \lambda_{(f-2)} - \rho_1 \right) 
    \ge \frac{\Delta_{f-1}}{4} \ge \frac{\tilde{\Delta}}{4}.
  \end{align*}
  Thus in both case, we have that $\lambda_{(f)} - \rho_1 \ge \frac{\tilde{\Delta}}{4}$ holds.

  Since the $\lambda_{(s)}$ values are monotonically non-increasing and lower-bounded by $\rho_1 + \frac{\tilde{\Delta}}{4}$, we have 
  \begin{align*}
    \max_{s}\ \sigma_{1} (\M_{\lambda_{(s)}}) = \sigma_{1} (\M_{\lambda_{(f)}}) = \frac{1}{\lambda_{(f)} - \rho_1 } \le \frac{4}{\tilde{\Delta} } =: \overline{\sigma}, 
  \end{align*}
  and 
  \begin{align*}
    \min_{s}\ \sigma_{1} (\M_{\lambda_{(s)}}) & = \sigma_{1} (\M_{\lambda_{(0)}}) = \frac{1}{\lambda_{(0)} - \rho_1 } = \frac{1}{ 1 + \tilde{\Delta} - \rho_1 } \\
    & \ge \frac{1}{ 1 + c_2 \Delta - \Delta } \ge 1 + (1 - c_2) \Delta \\
    & \ge 1 + \frac{1 - c_2}{c_2} \tilde{\Delta} =: \underline{\sigma},
  \end{align*}
  where the first inequality holds since by definition of $\Delta$ it follows that $\rho_1=\rho_2+\Delta \ge \Delta$.

  Then according to~\eqref{e:delta-s-1}, we can set  
  \begin{align*}  
    \te_s = \frac{ \underline{\sigma}^2 }{128 \overline{\sigma}} =  
    \frac{ \left( 1 + \frac{1 - c_2}{c_2} \tilde{\Delta} \right)^2  }{128 \cdot \frac{4}{\tilde{\Delta}}} 
    \ge \frac{ 1 }{128 \cdot \frac{4}{\tilde{\Delta}}} = \frac{\tilde{\Delta}}{512}.
  \end{align*}
  Because the initialization for minimizing $l_s(\uu)$ gives an initial suboptimality of $\te_s^{init}\le \frac{\sigma_1 \left( \M_{\lambda_{(s-1)}} \right)}{2} \le \frac{\overline{\sigma}}{2}=\frac{2}{\tilde{\Delta}}$, we achieve $\te_s$-suboptimality by requiring $\frac{\te_s^{init}}{\te_s} = \frac{1024}{\tilde{\Delta}^2}$.

  It remains to fulfill the requirement~\eqref{e:epsilon-1} for \textbf{repeat-until}. 
  Note that the condition numbers are bounded throughout:
  \begin{align} 
    \kappa_{\lambda_{(s)}} & = \frac{\lambda_{(s)} - \rho_d}{\lambda_{(s)} - \rho_1} \le \frac{\lambda_{(s)}}{\lambda_{(s)} - \rho_1} = 1 + \frac{\rho_1}{\lambda_{(s)} - \rho_1} \nonumber \\
    & \le   1 + \rho_1 \cdot \sigma_{1} (\M_{\lambda_{(s)}}) 
    \le 1 + \frac{4}{\tilde{\Delta}}
    \le \frac{5}{\tilde{\Delta}}.   \label{e:kappa-f}
  \end{align}
  We can then bound the ratio in~\eqref{e:epsilon-1} for all least squares subproblems:
  \begin{align*}
    \frac{\epsilon_t^{init}}{\epsilon_t} \le \frac{1024 \cdot 25}{\tilde{\Delta}^2} \left( \frac{ (10 / \tilde{\Delta})^{m_1}}{ 9 / \tilde{\Delta} } \right)^2 \le \frac{32 \cdot 10^{2 m_1+1}}{\tilde{\Delta}^{2 m_1}}.
  \end{align*}
\end{proof}

\section{Proof of Lemma~\ref{lem:shift-and-invert-for-loop}}

\begin{proof}
  Observe that for $\lambda= \rho_1 + c ( \rho_1 - \rho_2)$, we have
  \begin{align*}
    \frac{\sigma_1 (\M_{\lambda})}{ \sigma_1 (\M_{\lambda}) - \sigma_2 (\M_{\lambda}) } 
    & = \frac{\frac{1}{\lambda - \rho_1}}{\frac{1}{\lambda - \rho_1}  - \frac{1}{\lambda-\rho_2} } = \frac{\lambda - \rho_2}{\rho_1 - \rho_2} \\
    & = \frac{ \rho_1 + c (\rho_1 - \rho_2) - \rho_2}{\rho_1 - \rho_2} = c+1
  \end{align*}
  and 
  \begin{align*}
    \frac{3 \sigma_1 (\M_{\lambda}) + \sigma_2 (\M_{\lambda})}{\sigma_1 (\M_{\lambda}) + 3 \sigma_1 (\M_{\lambda})}
    = \frac{\frac{3}{\lambda - \rho_1} + \frac{1}{\lambda - \rho_2}}{\frac{1}{\lambda - \rho_1} + \frac{3}{\lambda - \rho_2}} = \frac{4c+3}{4c+1}.
  \end{align*}
  
  In view of~\eqref{e:delta-s-interval}, we have $\lambda_{(f)} - \rho_1 \le \frac{3}{2} \tilde{\Delta} \le  \frac{3 c_2}{2} {\Delta} \le  \frac{3}{2} {\Delta}$, \ie, $c \le \frac{3}{2}$ for $\lambda_{(f)}$. 
  As a result, we can apply Lemma~\ref{lem:iteration-complexity-inexact-power-method} (accurate regime) with $\gamma=\frac{9}{7}$, and we are guaranteed to achieve the desired alignment with the specified $m_2$. The choice of $\frac{\epsilon_t^{init}}{\epsilon_t}$ follows from an application of Lemma~\ref{lem:ft_ratio} (accurate regime) with $\frac{\beta_1^2}{\left( \beta_1 - \beta_2 \right)^2} \le \frac{25}{4}$.
\end{proof}

\section{Details of experiments on networks}
\label{sec:statistics-realdata}

Table~\ref{tab:data} gives the statistics of the real-world networks data used in our experiments.

 \begin{table}[h]
 \begin{center}
 \caption{List of network datasets used in the experiments.}
 \label{tab:data}
 \begin{tabular}{|c|c|c|}\hline
 Name  & $\#$Nodes & $\#$Edges  
 \\
 \hline\hline
 amazon0601        & 403,394  & 3,387,388 \\
 com-Youtube        & 1,134,890 & 2,987,624 \\
 email-Enron             &36,692  &183,831    \\
 email-EuAll       &265,214  &420,045       \\
 p2p-Gnutella31    &62,586  & 147,892  \\
 roadNet-CA        & 1,965,206 & 2,766,607	 \\
 roadNet-PA        & 1,379,917 & 1,921,660	 \\
 soc-Epinions1     & 75,879         & 508,837            \\
 web-BerkStan      & 685,230    & 7,600,595  \\
 web-Google       & 875,713  & 5,105,039 \\
 web-NotreDame    & 325,729	         &  	1,497,134       \\
 wiki-Talk        & 2,394,385 & 5,021,410 \\
 \hline
 \end{tabular}
 \end{center}
 \end{table}

For power method, CPM and SGCD, we adopt the C++ implementation 
 of~\citet{Lei_16a} and implement our algorithms alongside theirs, and run all experiments on a Linux machine with Intel i7 CPU of frequency 3.20GHz. We use $4$ passes of coordinate updates for solving least squares problems and set $\tilde{\Delta} = 0.05$ in Algorithm 1. 